\documentclass[12pt]{amsart}
\usepackage{amssymb}
\usepackage[all]{xy}

\newtheorem{theorem}{Theorem}[section]
\newtheorem{definition}[theorem]{Definition}
\newtheorem{proposition}[theorem]{Proposition}
\newtheorem{lemma}[theorem]{Lemma}

\begin{document}

\title{Quantum isometries of noncommutative polygonal spheres}

\author{Teodor Banica}
\address{T.B.: Department of Mathematics, Cergy-Pontoise University, 95000 Cergy-Pontoise, France. {\tt teodor.banica@u-cergy.fr}}


\begin{abstract}
The real sphere $S^{N-1}_\mathbb R$ appears as increasing union, over $d\in\{1,\ldots,N\}$, of its ``polygonal'' versions $S^{N-1,d-1}_\mathbb R=\{x\in S^{N-1}_\mathbb R|x_{i_0}\ldots x_{i_d}=0,\forall i_0,\ldots,i_d\ {\rm distinct}\}$. Motivated by general classification questions for the undeformed noncommutative spheres, smooth or not, we study here the quantum isometries of $S^{N-1,d-1}_\mathbb R$, and of its various noncommutative analogues, obtained via liberation and twisting. We discuss as well a complex version of these results, with $S^{N-1}_\mathbb R$ replaced by the complex sphere $S^{N-1}_\mathbb C$.
\end{abstract}

\maketitle

\section*{Introduction}

Goswami has shown in \cite{go1} that any noncommutative compact Riemannian manifold $X$ has a quantum isometry group $G^+(X)$. While the classical, connected manifolds cannot have genuine quantum isometries \cite{gjo}, the situation changes when looking at manifolds which are (1) disconnected, or (2) not smooth, or (3) not classical.

The fact that a disconnected manifold can have indeed quantum isometries is well-known, and goes back to Wang's paper \cite{wa2}, where a free analogue $S_N^+$ of the symmetric group $S_N$, acting on the $N$-point space $X_N=\{1,\ldots,N\}$, was constructed. For non-smooth (connected) manifolds this is a relatively new discovery, due to Huang \cite{hua}, the simplest example here being the action of $S_N^+$ on the union $Y_N=\bigcup_{i=1}^N[0,1]^{(i)}$ of the $N$ copies of the $[0,1]$-segment on the coordinate axes of $\mathbb R^N$. Finally, for the non-classical manifolds this is once again well-known, since \cite{go1}, a basic example here being the action of the free quantum group $O_N^+$ on the free real sphere $S^{N-1}_{\mathbb R,+}$, discussed in \cite{bgo}.

Generally speaking, understanding what exact geometric features of $X$ allow the existence of genuine quantum group actions is an open question. In view of the above results and examples, the answer probably involves a subtle mixture of non-connectedness, and non-smoothness, and non-commutativity, which remains yet to be determined.

The present paper is a continuation of \cite{ban}, \cite{bgo}, where we proposed the framework of ``undeformed noncommutative spheres'', and their submanifolds, as a reasonably general setting for investigating various quantum isometry phenomena. We will study here certain non-smooth versions of $S^{N-1}_\mathbb R$, and their various noncommutative analogues.

More precisely, we will be interested in the ``polygonal spheres'', and their noncommutative analogues appearing via liberation and twisting. The polygonal spheres are real algebraic manifolds, depending on integers $1\leq d\leq N$, defined as follows:
$$S^{N-1,d-1}_\mathbb R=\left\{x\in S^{N-1}_\mathbb R\Big|x_{i_0}\ldots x_{i_d}=0,\forall i_0,\ldots,i_d\ {\rm distinct}\right\}$$

This type of construction applies as well to the noncommutative versions of $S^{N-1}_\mathbb R$ constructed in \cite{ban}, \cite{bgo}. The cases $d=1,2$ are of particular interest, because we can recover in this way some key examples from \cite{ban}, originally dismissed there because of their non-smoothness. We have in fact 9 basic polygonal spheres, as follows:
$$\xymatrix@R=12mm@C=12mm{
S^{N-1}_\mathbb R\ar[r]&S^{N-1}_{\mathbb R,*}\ar[r]&S^{N-1}_{\mathbb R,+}\\
S^{N-1,1}_\mathbb R\ar[r]\ar[u]&S^{N-1,1}_{\mathbb R,*}\ar[r]\ar[u]&\bar{S}^{N-1}_{\mathbb R,*}\ar[u]\\
S^{N-1,0}_\mathbb R\ar[r]\ar[u]&\bar{S}^{N-1,1}_\mathbb R\ar[r]\ar[u]&\bar{S}^{N-1}_\mathbb R\ar[u]}$$

Here all the maps are inclusions. The 3 spheres on top are those in \cite{bgo}, the 3 spheres on the right are their twists, introduced in \cite{ban}, with the free sphere $S^{N-1}_{\mathbb R,+}$ being equal to its own twist, and the 4 spheres at bottom left appear as intersections.

We will first perform an axiomatic study of these 9 spheres, with some noncommutative algebraic geometry results, of diagrammatic type, extending those in \cite{ban}, \cite{bgo}. We will prove then that the corresponding quantum isometry groups are as follows:
$$\xymatrix@R=12mm@C=16mm{
O_N\ar[r]&O_N^*\ar[r]&O_N^+\\
H_N\ar[r]\ar[u]&H_N^{[\infty]}\ar[r]\ar[u]&\bar{O}_N^*\ar[u]\\
H_N^+\ar[r]\ar[u]&H_N\ar[r]\ar[u]&\bar{O}_N\ar[u]}$$

Here the 5 results on top and at right are known from \cite{ban}, \cite{bgo}. The 4 new results, at bottom left, concern the hyperoctahedral group $H_N$, and its versions $H_N^+,H_N^{[\infty]}$ from \cite{bbc}, \cite{bcs}. The proof uses methods from \cite{ban}, \cite{bgo}, \cite{bhg}, \cite{bdu}, \cite{rw2}, and some ad-hoc tricks.

We have as well a complex version of these results, concerning the 9 complex analogues of the above spheres and quantum groups, which once again extends some previous findings from \cite{ban}. We refer to the body of the paper for the precise statements of our results, and to the final section below for a list of questions raised by the present work.

The paper is organized as follows: in 1-2 we introduce the real polygonal spheres, in 3-4 we study their quantum isometries, and in 5-6 we state and prove our main results, we discuss the complex extensions, and we end with a few concluding remarks.

\medskip

\noindent {\bf Acknowledgements.} I would like to thank Jean-Marie Lescure for a useful discussion, and an anonymous referee for valuable suggestions. This work was partly supported by the NCN grant 2012/06/M/ST1/00169.

\section{Noncommutative spheres}

According to \cite{bgo}, the free analogue $S^{N-1}_{\mathbb R,+}$ of the real sphere $S^{N-1}_\mathbb R$ is the noncommutative real manifold whose coordinates $x_1,\ldots,x_N$ are subject to the condition $\sum_ix_i^2=1$. To be more precise, $S^{N-1}_{\mathbb R,+}$ is the abstract spectrum of the following universal $C^*$-algebra:
$$C(S^{N-1}_{\mathbb R,+})=C^*\left(x_1,\ldots,x_N\Big|x_i=x_i^*,x_1^2+\ldots+x_N^2=1\right)$$

We will be interested in what follows in various ``subspheres'' of $S^{N-1}_{\mathbb R,+}$. As explained in \cite{bgo}, besides $S^{N-1}_\mathbb R$, another fundamental example is the half-liberated sphere $S^{N-1}_{\mathbb R,*}$, which appears as an intermediate object, $S^{N-1}_\mathbb R\subset S^{N-1}_{\mathbb R,*}\subset S^{N-1}_{\mathbb R,+}$. Moreover, as explained in \cite{ban}, we have 2 more basic spheres obtained by twisting, $\bar{S}^{N-1}_\mathbb R\subset\bar{S}^{N-1}_{\mathbb R,*}\subset S^{N-1}_{\mathbb R,+}$.

Here is the precise definition of the 3 extra spheres:

\begin{definition}
The subspheres $\bar{S}^{N-1}_\mathbb R,S^{N-1}_{\mathbb R,*},\bar{S}^{N-1}_{\mathbb R,*}\subset S^{N-1}_{\mathbb R,+}$ are constructed by imposing the following conditions on the standard coordinates $x_1,\ldots,x_N$:
\begin{enumerate}
\item $\bar{S}^{N-1}_\mathbb R$: $x_ix_j=-x_jx_i$, for any $i\neq j$.

\item $S^{N-1}_{\mathbb R,*}$: $x_ix_jx_k=x_kx_jx_i$, for any $i,j,k$.

\item $\bar{S}^{N-1}_{\mathbb R,*}$: $x_ix_jx_k=-x_kx_jx_i$ for any $i,j,k$ distinct, $x_ix_jx_k=x_kx_jx_i$ otherwise.
\end{enumerate}
\end{definition}

The fact that we have indeed $\bar{S}^{N-1}_\mathbb R\subset\bar{S}^{N-1}_{\mathbb R,*}$ comes from $abc=-bac=bca=-cba$ for $a,b,c\in\{x_i\}$ distinct, and $aab=-aba=baa$ for $a,b\in\{x_i\}$ distinct, where $x_1,\ldots,x_N$ are the standard coordinates on $\bar{S}^{N-1}_\mathbb R$. In addition, it is known that the inclusions $S^{N-1}_\mathbb R\subset S^{N-1}_{\mathbb R,*}\subset S^{N-1}_{\mathbb R,+}$ and $\bar{S}^{N-1}_\mathbb R\subset\bar{S}^{N-1}_{\mathbb R,*}\subset S^{N-1}_{\mathbb R,+}$ are all proper at $N\geq 3$. See \cite{ban}.

As pointed out in \cite{ban}, when intersecting twisted and untwisted spheres, non-smooth manifolds can appear. More precisely, $S^{N-1}_\mathbb R\cap\bar{S}^{N-1}_{\mathbb R,*}$ consists by definition of the points $x\in S^{N-1}_\mathbb R$ having the property $x_ix_jx_k=0$ for any $i,j,k$ distinct, and is therefore a union of $\binom{N}{2}$ copies of the unit circle $\mathbb T$, which is not smooth. See \cite{ban}.

In what follows we will enlarge the formalism in \cite{ban}, as to cover as well these intersections, originally dismissed there, but which are quite interesting. First, we have:

\begin{proposition}
The $5$ main spheres, and the intersections between them, are
$$\xymatrix@R=12mm@C=12mm{
S^{N-1}_\mathbb R\ar[r]&S^{N-1}_{\mathbb R,*}\ar[r]&S^{N-1}_{\mathbb R,+}\\
S^{N-1,1}_\mathbb R\ar[r]\ar[u]&S^{N-1,1}_{\mathbb R,*}\ar[r]\ar[u]&\bar{S}^{N-1}_{\mathbb R,*}\ar[u]\\
S^{N-1,0}_\mathbb R\ar[r]\ar[u]&\bar{S}^{N-1,1}_\mathbb R\ar[r]\ar[u]&\bar{S}^{N-1}_\mathbb R\ar[u]}$$
where $\dot{S}^{N-1,d-1}_{\mathbb R,\times}\subset\dot{S}^{N-1}_{\mathbb R,\times}$ is obtained by assuming $x_{i_0}\ldots x_{i_d}=0$, for $i_0,\ldots,i_d$ distinct.
\end{proposition}

\begin{proof}
We must prove that the 4-diagram obtained by intersecting the 5 main spheres coincides with the 4-diagram appearing at bottom left in the statement:
$$\xymatrix@R=13mm@C=13mm{
S^{N-1}_\mathbb R\cap\bar{S}^{N-1}_{\mathbb R,*}\ar[r]&S^{N-1}_{\mathbb R,*}\cap\bar{S}^{N-1}_{\mathbb R,*}\\
S^{N-1}_\mathbb R\cap\bar{S}^{N-1}_\mathbb R\ar[r]\ar[u]&S^{N-1}_{\mathbb R,*}\cap\bar{S}^{N-1}_\mathbb R\ar[u]}
\ \ \xymatrix@R=7mm@C=1mm{&\\=\\&}
\xymatrix@R=13mm@C=13mm{
S^{N-1,1}_\mathbb R\ar[r]&S^{N-1,1}_{\mathbb R,*}\\
S^{N-1,0}_\mathbb R\ar[r]\ar[u]&\bar{S}^{N-1,1}_\mathbb R\ar[u]}$$

But this is clear, because combining the commutation and anticommutation relations leads to the vanishing relations defining spheres of type $\dot{S}^{N-1,d-1}_{\mathbb R,\times}$. More precisely:

(1) $S^{N-1}_\mathbb R\cap\bar{S}^{N-1}_\mathbb R$ consists of the points $x\in S^{N-1}_\mathbb R$ satisfying $x_ix_j=-x_jx_i$ for $i\neq j$. Since $x_ix_j=x_jx_i$, this latter relation reads $x_ix_j=0$ for $i\neq j$, which means $x\in S^{N-1,0}_\mathbb R$.

(2) $S^{N-1}_\mathbb R\cap\bar{S}^{N-1}_{\mathbb R,*}$ consists of the points $x\in S^{N-1}_\mathbb R$ satisfying $x_ix_jx_k=-x_kx_jx_i$ for $i,j,k$ distinct. Once again by commutativity, this relation is equivalent to $x\in S^{N-1,1}_\mathbb R$.

(3) $S^{N-1}_{\mathbb R,*}\cap\bar{S}^{N-1}_\mathbb R$ is obtained from $\bar{S}^{N-1}_\mathbb R$ by imposing to the standard coordinates the half-commutation relations $abc=cba$. On the other hand, we know from $\bar{S}^{N-1}_\mathbb R\subset \bar{S}^{N-1}_{\mathbb R,*}$ that the standard coordinates on $\bar{S}^{N-1}_\mathbb R$ satisfy $abc=-cba$ for $a,b,c$ distinct, and $abc=cba$ otherwise. Thus, the relations brought by intersecting with $S^{N-1}_{\mathbb R,*}$ reduce to the relations $abc=0$ for $a,b,c$ distinct, and so we are led to the sphere $\bar{S}^{N-1,1}_\mathbb R$.

(4) $S^{N-1}_{\mathbb R,*}\cap\bar{S}^{N-1}_{\mathbb R,*}$ is obtained from $\bar{S}^{N-1}_{\mathbb R,*}$ by imposing the relations $abc=-cba$ for $a,b,c$ distinct, and $abc=cba$ otherwise. Since we know that $abc=cba$ for any $a,b,c$, the extra relations reduce to $abc=0$ for $a,b,c$ distinct, and so we are led to $S^{N-1,1}_{\mathbb R,*}$.
\end{proof}

Let us find now a suitable axiomatic framework for the 9 spheres in Proposition 1.2. We denote by $P(k,l)$ the set of partitons between an upper row of $k$ points, and a lower row of $l$ points, we set $P=\bigcup_{kl}P(k,l)$, and we denote by $P_{even}\subset P$ the subset of partitions having all the blocks of even size. Observe that $P_{even}(k,l)=\emptyset$ for $k+l$ odd.

We use the fact, from \cite{ban}, that there is a signature map $\varepsilon: P_{even}\to\{-1,1\}$, extending the usual signature of permutations, $\varepsilon:S_\infty\to\{-1,1\}$. This map is obtained by setting $\varepsilon(\pi)=(-1)^c$, where $c\in\mathbb N$ is the number of switches between neighbors required for making $\pi$ noncrossing, and which can be shown to be well-defined modulo 2.

We have the following definition, once again from \cite{ban}:

\begin{definition}
Given variables $x_1,\ldots,x_N$, any permutation $\sigma\in S_k$ produces two collections of relations between these variables, as follows:
\begin{enumerate}
\item Untwisted relations: $x_{i_1}\ldots x_{i_k}=x_{i_{\sigma(1)}}\ldots x_{i_{\sigma(k)}}$, for any $i_1,\ldots,i_k$.

\item Twisted relations: $x_{i_1}\ldots x_{i_k}=\varepsilon\left(\ker(^{\,\,\,i_1\ \,\ldots\ \,i_k}_{i_{\sigma(1)}\ldots i_{\sigma(k)}})\right)x_{i_{\sigma(1)}}\ldots x_{i_{\sigma(k)}}$, for any $i_1,\ldots,i_k$.
\end{enumerate}
The untwisted relations are denoted $\mathcal R_\sigma$, and the twisted ones are denoted $\bar{\mathcal R}_\sigma$.
\end{definition}

Observe that the relations $\mathcal R_\sigma$ are trivially satisfied for the standard coordinates on $S^{N-1}_\mathbb R$, for any $\sigma\in S_k$. A twisted analogue of this fact holds, in the sense that the standard coordinates on $\bar{S}^{N-1}_\mathbb R$ satisfy the relations $\bar{\mathcal R}_\sigma$, for any $\sigma\in S_k$. Indeed, by anticommutation we must have a formula of type $x_{i_1}\ldots x_{i_k}=\pm x_{i_{\sigma(1)}}\ldots x_{i_{\sigma(k)}}$, and the sign $\pm$ obtained in this way is precisely the one given above, $\pm=\varepsilon\left(\ker(^{\,\,\,i_1\ \,\ldots\ \,i_k}_{i_{\sigma(1)}\ldots i_{\sigma(k)}})\right)$. See \cite{ban}.

Finally, we agree as in \cite{ban} to distinguish the untwisted and twisted cases by using a dot symbol, which is null in the untwisted case, and is a bar in the twisted case.

We have now all the needed ingredients for axiomatizing the various spheres:

\begin{definition}
We have $3$ types of noncommutative spheres $S\subset S^{N-1}_{\mathbb R,+}$, as follows:
\begin{enumerate}
\item Monomial: $\dot{S}^{N-1}_{\mathbb R,E}$, with $E\subset S_\infty$, obtained via the relations $\{\dot{\mathcal R}_\sigma|\sigma\in E\}$.

\item Mixed monomial: $S^{N-1}_{\mathbb R,E,F}=S^{N-1}_{\mathbb R,E}\cap\bar{S}^{N-1}_{\mathbb R,F}$, with $E,F\subset S_\infty$.

\item Polygonal: $S^{N-1,d-1}_{\mathbb R,E,F}=S^{N-1}_{\mathbb R,E,F}\cap S^{N-1,d-1}_{\mathbb R,+}$, with $E,F\subset S_\infty$, and $d\in\{1,\ldots,N\}$.
\end{enumerate}
\end{definition}

Here the subsphere $S^{N-1,d-1}_{\mathbb R,+}\subset S^{N-1}_{\mathbb R,+}$ appearing in (3) is constructed as in Proposition 1.2 above, by imposing the relations $x_{i_0}\ldots x_{i_d}=0$, for $i_0,\ldots,i_d$ distinct.

With the above notions, we cover all spheres appearing so far. More precisely, the 5 spheres in \cite{ban} are monomial, the 9 spheres in Proposition 1.2 are mixed monomial, and the polygonal sphere formalism covers all the examples given so far in this paper.

Observe that the set of mixed monomial spheres is closed under intersections. The same holds for the set of polygonal spheres, because we have the following formula:
$$S^{N-1,d-1}_{\mathbb R,E,F}\cap S^{N-1,d'-1}_{\mathbb R,E',F'}=S^{N-1,min(d,d')-1}_{\mathbb R,E\cup E',F\cup F'}$$

Let us try now to understand the structure of the various types of noncommutative spheres. We call a group of permutations $G\subset S_\infty$ filtered if, with $G_k=G\cap S_k$, we have $G_k\times G_l\subset G_{k+l}$, for any $k,l$. We use the following simple fact, coming from \cite{ban}:

\begin{proposition}
The various spheres can be parametrized by groups, as follows:
\begin{enumerate}
\item Monomial case: $\dot{S}^{N-1}_{\mathbb R,G}$, with $G\subset S_\infty$ filtered group.

\item Mixed monomial case: $S^{N-1}_{\mathbb R,G,H}$, with $G,H\subset S_\infty$ filtered groups.

\item Polygonal case: $S^{N-1,d-1}_{\mathbb R,G,H}$, with $G,H\subset S_\infty$ filtered groups, and $d\in\{1,\ldots,N\}$.
\end{enumerate}
\end{proposition}

\begin{proof}
As explained in \cite{ban}, in order to prove (1) for a monomial sphere $S=\dot{S}_{\mathbb R,E}$, we can take $G\subset S_\infty$ to be the set of permutations $\sigma\in S_\infty$ having the property that the relations $\dot{\mathcal R}_\sigma$ hold for the standard coordinates of $S$. We have then $E\subset G$, we have as well $S=\dot{S}^{N-1}_{\mathbb R,G}$, and the fact that $G$ is a filtered group is clear as well. See \cite{ban}.

Regarding now (2) and (3), these follow from (1), by taking intersections.
\end{proof}

Let us write now the 9 main polygonal spheres as in Proposition 1.5 (2). We recall from \cite{ban} that the permutations $\sigma\in S_\infty$ having the property that when labelling clockwise their legs $\circ\bullet\circ\bullet\ldots$, and string joins a white leg to a black leg, form a filtered group, denoted $S_\infty^*\subset S_\infty$. This group comes from the half-liberation considerations in \cite{bve}, and its structure is very simple, $S_{2n}^*\simeq S_n\times S_n,S_{2n+1}^*\simeq S_n\times S_{n+1}$. See \cite{ban}.

We call a mixed monomial sphere parametrization $S=S^{N-1}_{\mathbb R,G,H}$ standard when both filtered groups $G,H\subset S_\infty$ are chosen to be maximal. In this case, Proposition 1.5 and its proof tell us that $G,H$ encode all the monomial relations which hold in $S$.

We have the following result, extending some previous findings from \cite{ban}:

\begin{theorem}
The standard parametrization of the $9$ main spheres is
$$\xymatrix@R=10mm@C=10mm{
S_\infty\ar@{.}[d]&S_\infty^*\ar@{.}[d]&\{1\}\ar@{.}[d]&G/H\\
S^{N-1}_\mathbb R\ar[r]&S^{N-1}_{\mathbb R,*}\ar[r]&S^{N-1}_{\mathbb R,+}&\{1\}\ar@{.}[l]\\
S^{N-1,1}_\mathbb R\ar[r]\ar[u]&S^{N-1,1}_{\mathbb R,*}\ar[r]\ar[u]&\bar{S}^{N-1}_{\mathbb R,*}\ar[u]&S_\infty^*\ar@{.}[l]\\
S^{N-1,0}_\mathbb R\ar[r]\ar[u]&\bar{S}^{N-1,1}_\mathbb R\ar[r]\ar[u]&\bar{S}^{N-1}_\mathbb R\ar[u]&S_\infty\ar@{.}[l]}$$
where $S_\infty^*\subset S_\infty$ is given by $S_{2n}^*\simeq S_n\times S_n,S_{2n+1}^*\simeq S_n\times S_{n+1}$.
\end{theorem}

\begin{proof}
The fact that we have parametrizations as in the statement is known to hold for the 5 main spheres from \cite{ban}, as explained there. For the remaining 4 spheres the result follows by intersecting, by using the following formula, valid for any $E,F\subset S_\infty$:
$$S^{N-1}_{\mathbb R,E,F}\cap S^{N-1}_{\mathbb R,E',F'}=S^{N-1}_{\mathbb R,E\cup E',F\cup F'}$$

In order to prove that the parametrizations are standard, we must compute the following two filtered groups, and show that we get the groups in the statement:
$$G=\{\sigma\in S_\infty|{\rm the\ relations\ }\mathcal R_\sigma\ {\rm hold\ over\ }X\}$$ 
$$H=\{\sigma\in S_\infty|{\rm the\ relations\ }\bar{\mathcal R}_\sigma\ {\rm hold\ over\ }X\}$$ 

As a first observation, by using the various inclusions between spheres, we just have to compute $G$ for the spheres on the bottom, and $H$ for the spheres on the left:
$$X=S^{N-1,0}_\mathbb R,\bar{S}^{N-1,1}_\mathbb R,\bar{S}^{N-1}_\mathbb R\implies G=S_\infty,S_\infty^*,\{1\}$$
$$X=S^{N-1,0}_\mathbb R,S^{N-1,1}_\mathbb R,S^{N-1}_\mathbb R\implies H=S_\infty,S_\infty^*,\{1\}$$

The results for $S^{N-1,0}_\mathbb R$ being clear, we are left with computing the remaining 4 groups, for the spheres $S^{N-1}_\mathbb R,\bar{S}^{N-1}_\mathbb R,S^{N-1,1}_\mathbb R,\bar{S}^{N-1,1}_\mathbb R$. The proof here goes as follows:

(1) $S^{N-1}_\mathbb R$. According to the definition of $H=(H_k)$, we have:
\begin{eqnarray*}
H_k
&=&\left\{\sigma\in S_k\Big|x_{i_1}\ldots x_{i_k}=\varepsilon\left(\ker(^{\,\,\,i_1\ \,\ldots\ \,i_k}_{i_{\sigma(1)}\ldots i_{\sigma(k)}})\right)x_{i_{\sigma(1)}}\ldots x_{i_{\sigma(k)}},\forall i_1,\ldots,i_k\right\}\\
&=&\left\{\sigma\in S_k\Big|\varepsilon\left(\ker(^{\,\,\,i_1\ \,\ldots\ \,i_k}_{i_{\sigma(1)}\ldots i_{\sigma(k)}})\right)=1,\forall i_1,\ldots,i_k\right\}\\
&=&\left\{\sigma\in S_k\Big|\varepsilon(\tau)=1,\forall\tau\leq\sigma\right\}\end{eqnarray*}

Now since for any $\sigma\in S_k,\sigma\neq1_k$, we can always find a partition $\tau\leq\sigma$ satisfying $\varepsilon(\tau)=-1$, we deduce that we have $H_k=\{1_k\}$, and so $H=\{1\}$, as desired.

(2) $\bar{S}^{N-1}_\mathbb R$. The proof of $G=\{1\}$ here is similar to the proof of $H=\{1\}$ in (1) above, by using the same combinatorial ingredient at the end.

(3) $S^{N-1,1}_\mathbb R$. By definition of $H=(H_k)$, a permutation $\sigma\in S_k$ belongs to $H_k$ when the following condition is satisfied, for any choice of the indices $i_1,\ldots,i_k$:
$$x_{i_1}\ldots x_{i_k}=\varepsilon\left(\ker(^{\,\,\,i_1\ \,\ldots\ \,i_k}_{i_{\sigma(1)}\ldots i_{\sigma(k)}})\right)x_{i_{\sigma(1)}}\ldots x_{i_{\sigma(k)}}$$

When $|\ker i|=1$ this formula reads $x_r^k=x_r^k$, which is true. When $|\ker i|\geq3$ this formula is automatically satisfied as well, because by using the relations $ab=ba$, and $abc=0$ for $a,b,c$ distinct, which both hold over $S^{N-1,1}_\mathbb R$, this formula reduces to $0=0$. Thus, we are left with studying the case $|\ker i|=2$. Here the quantities on the left $x_{i_1}\ldots x_{i_k}$ will not vanish, so the sign on the right must be 1, and we therefore have:
$$H_k=\left\{\sigma\in S_k\Big|\varepsilon(\tau)=1,\forall\tau\leq\sigma,|\tau|=2\right\}$$

Now by coloring the legs of $\sigma$ clockwise $\circ\bullet\circ\bullet\ldots$, the above condition is satisfied when each string of $\sigma$ joins a white leg to a black leg. Thus $H_k=S_k^*$, as desired.

(4) $\bar{S}^{N-1,1}_\mathbb R$. The proof of $G=S_\infty^*$ here is similar to the proof of $H=S_\infty^*$ in (3) above, by using the same combinatorial ingredient at the end.
\end{proof}

As a conclusion, the $5+4=9$ spheres from Proposition 1.2 come from the $3\times 3$ ways of selecting a pair of filtered groups $(G,H)$, among the basic filtered groups $\{1\},S_\infty^*,S_\infty$. This result, improving some previous findings from \cite{ban}, is the best one that we have.

\section{Uniqueness results}

In this section we discuss a number of conjectures, whose validity would improve the formalism in Theorem 1.6. These conjectures are all equivalent, as follows:

\begin{proposition}
The following are equivalent:
\begin{enumerate}
\item The $3$ spheres in \cite{bgo} are the only untwisted monomial ones.

\item The $5$ spheres in \cite{ban} are the only monomial ones.

\item The $9$ spheres in Theorem 1.6 are the only mixed monomial ones.
\end{enumerate}
\end{proposition}

\begin{proof}
These equivalences are all clear, with $(1)\implies(2)\implies(3)$ being obtained by intersecting, and with $(3)\implies(2)\implies(1)$ being obtained by restricting.
\end{proof}

These conjectures belong a priori to operator theory/algebras, and more specifically to a branch that could be called ``noncommutative algebraic geometry, with positivity'', that we are trying to develop in this paper. Our claim here would be that there might be a purely combinatorial way of solving them. We have the following definition:

\begin{definition}
Consider a filtered group of permutations, that is, a group $G\subset S_\infty$, $G=(G_k)$, satisfying $G_k\times G_l\subset G_{k+l}$ for any $k,l$. We call this group:
\begin{enumerate}
\item Saturated, if $G$ consists of all the permutations $\sigma\in S_k$ such that the relations $x_{i_1}\ldots x_{i_k}=x_{i_{\sigma(1)}}\ldots x_{i_{\sigma(k)}}$ hold over $S^{N-1}_{\mathbb R,G}$, for any $i_1,\ldots,i_k$.

\item Weakly saturated, if whenever $\sigma\in G_k$ satisfies $\sigma(i+1)=\sigma(i)\pm 1$, the permutation $\sigma^{(i,i+1)}\in S_{k-2}$ obtained by deleting $i,i+1$ and their images belongs to $G_{k-2}$.
\end{enumerate}
\end{definition}

It follows from Proposition 1.5 that we have a saturation operation $G\to\widetilde{G}$ for the filtered groups, which can be obtained by setting $S^{N-1}_{\mathbb R,G}=S^{N-1}_{\mathbb R,\widetilde{G}}$, with $\widetilde{G}\subset S_\infty$ chosen maximal. With this remark in hand, the conjecture in Proposition 2.1 (1) above simply states that there are exactly 3 saturated groups, namely $\{1\},S_\infty^*,S_\infty$. Observe that these 3 groups are indeed saturated, as a consequence of Theorem 1.6 above.

Regarding now the weak saturation, once again this produces an operation $G\to\bar{G}$ for the filtered groups. Indeed, given $G\subset S_\infty$ we can add to it all the permutations $\sigma^{(i,i+1)}$ appearing in Definition 2.2 (2), then consider the filtered group generated by $G$ and by these extra permutations, and then repeat the procedure, a finite or possibly countable number of times, until we obtain a weakly saturated group $\bar{G}$.

The interest in the above notions comes from:

\begin{proposition}
Any saturated group is weakly saturated. In particular, if the only weakly saturated groups are $\{1\},S_\infty^*,S_\infty$, then the conjectures in Proposition 2.1 hold.
\end{proposition}

\begin{proof}
Consider a saturated group $G\subset S_\infty$, and let $S=S^{N-1}_{\mathbb R,G}$ be the corresponding sphere. We must show that if $\sigma\in G_k$ satisfies $\sigma(i+1)=\sigma(i)\pm 1$, then $\sigma^{(i,i+1)}\in G_{k-2}$.

We know from $\sigma\in G_k$ that the relations $x_{i_1}\ldots x_{i_k}=x_{i_{\sigma(1)}}\ldots x_{i_{\sigma(k)}}$ hold over $S$. In the case $\sigma(i+1)=\sigma(i)+1$ these relations are of type $XabY=ZabT$, and by setting $a=b^*$ and summing over $a$ we obtain $XY=ZT$. But these are exactly the relations associated to the permutation $\sigma^{(i,i+1)}\in S_{k-2}$, and we deduce that we have $\sigma^{(i,i+1)}\in G_{k-2}$.

In the case $\sigma(i+1)=\sigma(i)-1$ the proof is similar. Indeed, the relations associated to $\sigma$ are now of type $XabY=ZbaT$, and once again by setting $a=b^*$, and by summing over $a$, we obtain $XY=ZT$, and we conclude that we have $\sigma^{(i,i+1)}\in G_{k-2}$.

Finally, the last assertion is clear from the above considerations.
\end{proof}

We have the following result, of interest in connection with Proposition 2.1:

\begin{proposition}
If a filtered group $G\subset S_\infty$, $G=(G_k)$ is weakly saturated and $|G_5|>1$, then $G$ must be one of the groups $S_\infty,S_\infty^*$.
\end{proposition}

\begin{proof}
Our claim, which will basically prove the result, is that at $k\leq5$ we have:
$$\sigma\in S_k\implies\exists\tau\in<1\otimes\sigma,\sigma\otimes 1>\subset S_{k+1},\exists i,\tau(i+1)=\tau(i)\pm1$$

We have no conceptual proof for this claim, so we will first discuss the cases $k=3,4$, following some previous work in \cite{ban}, and then we will discuss the case $k=5$:

\underline{Case $k=3$.} Here we just have to investigate the 3-cycles, and by symmetry we can restrict attention to the cycle $\sigma=(231)$. As explained in \cite{ban}, a standard $C^*$-algebra trick shows that the corresponding sphere collapses to $S^{N-1}_\mathbb R$. The point now is that this trick can be converted into a proof of the above claim. More precisely, we have $(1\otimes\sigma)(\sigma\otimes1)=(2143)$, which satisfies the requirements for $\tau$ in the above claim.

\underline{Case $k=4$.} Here, as explained in \cite{ban}, for 22 of the 24 permutations $\sigma\in S_4$, the above claim holds, with $\tau=\sigma$. The remaining 2 permutations are $\sigma_1=(3412)$ and $\sigma_2=(2413)$. The point now is that we have $(1\otimes\sigma_1)(\sigma_1\otimes1)=(52143)$ and $\sigma_2^2=(4321)$, which both satisfy the requirements for $\tau$ in the above claim. See \cite{ban}.

\underline{Case $k=5$.} We have to study the 120 elements $\sigma\in S_5$, and best here is to consider the corresponding group $<\sigma>\subset S_5$, which is $G=\mathbb Z_s$ with $s=1,2,3,4,5,6$. At $s=1$ the result is clear, at $s=3,4,5$ what happens is that we can always find $\tau\in G$ satisfying $\tau(i+1)=\tau(i)\pm1$, for some $i$, and at $s=6$ the result follows from the $s=3$ result. Thus we are left with the case $s=2$. Here the cycle structure of $\sigma$ can be either $(2111)$, where the result is clear, or $(221)$, which is the case left. But here $\sigma$ must appear from one of the elements $(2143),(4321),(3412)\in S_4$ by adding a ``fixed point''. When this fixed point is at right or at left, the result is clear, so by symmetry it remains to study the 2 cases where this fixed point is either in the middle, or at left of the middle point. Thus we have $3\times2=6$ cases to be investigated, and 5 of these cases are trivial, in the sense that $\sigma$ itself satisfies $\sigma(i+1)=\sigma(i)\pm1$ for some $i$. The remaining case is $\sigma=(42513)$, and here $(1\otimes\sigma)(\sigma\otimes1)=(435621)$, which satisfies the requirements for $\tau$ in the above claim.

Thus we are done with the proof of the above claim. The point now is that, given $G\subset S_\infty$ as in the statement, we can pick $\sigma\in G_5-\{1_5\}$, and apply to it the above claim, perhaps several times, until we obtain either the basic crossing $(12)\in S_2$, or the half-liberated partition $(321)\in S_3$. We deduce from this that $G$ must be generated by one of these two partitions, and so we have $G=S_\infty$ or $G=S_\infty^*$, as desired.
\end{proof}

By combining now Proposition 2.3 and Proposition 2.4, we deduce:

\begin{theorem}
The conjectures in Proposition 2.1 hold, provided that the spheres in question are generated by relations coming from permutations $\sigma\in S_k$ with $k\leq5$.
\end{theorem}

\begin{proof}
Proposition 2.3 and Proposition 2.4 tell us that at $k\leq5$ we have: 
$$\sigma\in S_k,\sigma\neq 1_k\implies S^{N-1}_{\mathbb R,\sigma}\in\{S^{N-1}_\mathbb R,S^{N-1}_{\mathbb R,*}\}$$

Thus the conjecture in Proposition 2.1 (1) holds under the $k\leq5$ assumption. The statements coming from Proposition 2.1 (2) and (3) follow as well, by intersecting.
\end{proof}

We believe that Proposition 2.4 should hold under the assumption $G\neq\{1\}$, therefore proving the conjectures in Proposition 2.1, but we were unable so far to extract something conceptual from the above proof, which would extend from $k\leq5$ to $k\in\mathbb N$. 

As a second piece of evidence for the conjectures in Proposition 2.1, we can try to intersect an arbitrary untwisted monomial sphere $S=S^{N-1}_{\mathbb R,E}$ with the 3 untwisted monomial spheres, or with the 5 monomial spheres, or with the 9 mixed monomial spheres, and see if we get indeed the results predicted by $S\in\{S^{N-1}_\mathbb R,S^{N-1}_{\mathbb R,*},S^{N-1}_{\mathbb R,+}\}$.

There are many interesting statements here, and as an example, we have:

\begin{proposition}
For any $F\subset S_\infty$ we have a formula as follows
$$S^{N-1}_\mathbb R\cap\bar{S}^{N-1}_{\mathbb R,F}=S^{N-1,d-1}_\mathbb R$$
for a certain number $d\in\{1,\ldots,N\}$. In addition, we have $d\in\{1,2,N\}$.
\end{proposition}

\begin{proof}
We can assume $F=\{\sigma\}$, with $\sigma\in S_k$, $\sigma\neq 1_k$.

(1) The anticommutation relations, when compared to the corresponding commutation relations, translate into relations of type $a_{i_1}\ldots a_{i_k}=0$, for certain indices $i$. Since we can permute the terms, and we can also replace $x_i^r\to x_i$ for any $r\geq2$, we are led to relations of type $x_{i_0}\ldots x_{i_r}=0$, for any $i_0,\ldots,i_r$ distinct. Now since the spheres $S^{N-1,r-1}_\mathbb R$ form an increasing sequence, by setting $d=\min(r)$ we obtain the formula in the statement.

(2) We use the defining formulae for $\bar{S}^{N-1}_{\mathbb R,\sigma}$, which are:
$$x_{i_1}\ldots x_{i_k}=\varepsilon\left(\ker\begin{pmatrix}i_1&\ldots&i_k\\ i_{\sigma(1)}&\ldots&i_{\sigma(k)}\end{pmatrix}\right)x_{i_{\sigma(1)}}\ldots x_{i_{\sigma(k)}}$$

By comparing with the commutation relation $x_{i_1}\ldots x_{i_k}=x_{i_{\sigma(1)}}\ldots x_{i_{\sigma(k)}}$, valid for the classical sphere $S^{N-1}_\mathbb R$, we conclude that the intersection $S^{N-1}_\mathbb R\cap \bar{S}^{N-1}_{\mathbb R,\sigma}$ consists of the points $x\in S^{N-1}_\mathbb R$ which are subject to the following relations:
$$\varepsilon\left(\ker\begin{pmatrix}i_1&\ldots&i_k\\ i_{\sigma(1)}&\ldots&i_{\sigma(k)}\end{pmatrix}\right)=-1\implies x_{i_1}\ldots x_{i_k}=0$$

In other words, given our permutation $\sigma\in S_k$, we can consider all the partitions $\pi\leq\sigma$, obtained by collapsing blocks. The partitions satisfying $\varepsilon(\pi)=1$ don't produce new relations, and the partitions satisfying $\varepsilon(\pi)=-1$ produce the following relations, where $r=|\pi|$ comes from the compression procedure explained in the proof of (1) above:
$$\ker\begin{pmatrix}i_1&\ldots&i_k\\ i_{\sigma(1)}&\ldots&i_{\sigma(k)}\end{pmatrix}=\pi\implies x_{i_1}\ldots x_{i_r}=0$$

We use now the fact that $\sigma\in S_\infty$, $\sigma\neq 1_1,1_2,1_3,\ldots$ implies $\exists\pi\leq\sigma$, $\varepsilon(\pi)=-1,|\pi|\leq3$, which comes by selecting a suitable crossing for $\sigma$, and then by collapsing all the other strings to a single block. Thus $d+1=\min(r)$ satisfies $d\in\{2,3\}$, and we are done.
\end{proof}

As a last comment, a useful ingredient for dealing with the conjectures in Proposition 2.1 would be a good diagrammatic framework for the polygonal spheres. Observe that all the relations that we need are of the following type, with $\alpha,\beta\in\{-,1,0,1\}$:
$$x_{i_1}\ldots x_{i_k}=\begin{cases}
\alpha\cdot x_{i_{\sigma(1)}}\ldots x_{i_{\sigma(k)}}&{\rm if}\ \varepsilon=1\\
\beta\cdot x_{i_{\sigma(1)}}\ldots x_{i_{\sigma(k)}}&{\rm if}\ \varepsilon=-1
\end{cases}$$

Here the number $\varepsilon=\pm1$ on the right is by definition given by:
$$\varepsilon=\varepsilon\left(\ker\begin{pmatrix}i_1&\ldots&i_k\\ i_{\sigma(1)}&\ldots&i_{\sigma(k)}\end{pmatrix}\right)$$

Thus the diagrams that we need are a priori the usual permutations, colored in $3\times 3=9$ ways, according to the values of $(\alpha,\beta)$. It is quite unclear, however, on how to turn this idea into an efficient computational tool, that can solve our conjectures.

\section{Affine actions}

We discuss now the computation of the quantum isometry groups of our spheres. We use the compact quantum group formalism developed by Woronowicz in \cite{wo1}, \cite{wo2}.

There are several ways of defining quantum isometries, depending on the type of manifold involved, see \cite{bsk}, \cite{bgs}, \cite{chi}, \cite{con}, \cite{go1}, \cite{go2}, \cite{qsa}, \cite{thi}, \cite{wa1}, \cite{wa2}. In what follows we use an algebraic geometry approach, inspired from Goswami's paper \cite{go2}. 

Assume that we are given an algebraic manifold $X\subset S^{N-1}_{\mathbb R,+}$, in the sense that $X$ appears via a presentation result as follows, for certain noncommutative polynomials $P_\alpha$:
$$C(X)=C(S^{N-1}_{\mathbb R,+})/<P_\alpha(x_1,\ldots,x_N)=0>$$

We say that a closed subgroup $G\subset O_N^+$ acts affinely on $X$ when we have a morphism of $C^*$-algebras $\Phi:C(X)\to C(G)\otimes C(X)$, given by $x_i\to\sum_ju_{ij}\otimes x_i$. Observe that such a morphism is automatically coassociative and counital, and unique. We have:

\begin{proposition}
Given an algebraic submanifold $X\subset S^{N-1}_{\mathbb R,+}$, we category of closed quantum subgroups $G\subset O_N^+$ acting affinely on $X$ has a universal object, $G^+(X)$.
\end{proposition}

\begin{proof}
Assume indeed that $X$ is defined by polynomials $P_\alpha$ as above. Our claim is that $G=G^+(X)$ appears as follows, where $X_i=\sum_ju_{ij}\otimes x_j\in C(O_N^+)\otimes C(X)$:
$$C(G)=C(O_N^+)/<P_\alpha(X_1,\ldots,X_N)=0>$$

In order to prove this claim, we have to clarify how the relations $P_\alpha(X_1,\ldots,X_N)=0$ are interpreted inside $C(O_N^+)$, and then show that $G$ is indeed a quantum group.

So, pick one of the defining polynomials, $P=P_\alpha$, and write it as follows:
$$P(x_1,\ldots,x_N)=\sum_r\alpha_r\cdot x_{i_1^r}\ldots x_{i_{s(r)}^r}$$

With $X_i=\sum_ju_{ij}\otimes x_j$ as above, we have the following formula:
$$P(X_1,\ldots,X_N)=\sum_r\alpha_r\sum_{j_1^r\ldots j_{s(r)}^r}u_{i_1^rj_1^r}\ldots u_{i_{s(r)}^rj_{s(r)}^r}\otimes x_{j_1^r}\ldots x_{j_{s(r)}^r}$$

Since the space spanned by the variables at right is finite dimensional, the relations $P(X_1,\ldots,X_N)=0$ correspond indeed to certain relations between the variables $u_{ij}$.

It order to show that $G$ is indeed a quantum group, consider the following elements:
$$u_{ij}^\Delta=\sum_ku_{ik}\otimes u_{kj}\quad,\quad
u_{ij}^\varepsilon=\delta_{ij}\quad,\quad
u_{ij}^S=u_{ji}$$

Now if we consider the associated elements $X_i^\gamma=\sum_ju_{ij}^\gamma\otimes x_j$, with $\gamma\in\{\Delta,\varepsilon,S\}$, then from the relations $P(X_1,\ldots,X_N)=0$ we deduce that we have:
$$P(X_1^\gamma,\ldots,X_N^\gamma)=(\gamma\otimes id)P(X_1,\ldots,X_N)=0$$

Thus, by using the universal property of $G$, we can construct morphisms of algebras mapping $u_{ij}\to u_{ij}^\gamma$ for any $\gamma\in\{\Delta,\varepsilon,S\}$, and this finishes the proof.
\end{proof}

As an illustration, we have the following statement, coming from \cite{bhg}, \cite{go2}:

\begin{proposition}
Assume that $X\subset S^{N-1}_\mathbb R$ is invariant under $x_i\to-x_i$, for any $i$.
\begin{enumerate}
\item If the coordinates $x_1,\ldots,x_N$ are linearly independent inside $C(X)$, then the group $G(X)=G^+(X)\cap O_N$ consists of the usual isometries of $X$.

\item In addition, in the case where the products of coordinates $\{x_ix_j|i\leq j\}$ are linearly independent inside $C(X)$, we have $G^+(X)=G(X)$.
\end{enumerate}
\end{proposition}

\begin{proof}
This follows from \cite{bhg}, \cite{go2}, the idea being as follows:

(1) The assertion here is well-known, $G(X)=G^+(X)\cap O_N$ being by definition the biggest subgroup $G\subset O_N$ acting affinely on $X$. We refer to \cite{go2} for details, and for a number of noncommutative extensions of this fact, with $G(X)$ replaced by $G^+(X)$.

(2) Here we must prove that, whenever we have a coaction $\Phi:C(X)\to C(G)\otimes C(X)$, given by $\Phi(x_i)=\sum_ju_{ij}\otimes x_j$, the variables $u_{ij}$ commute. But this follows by using a strandard trick, from \cite{bhg}, that we will briefly recall now. We can write:
$$\Phi([x_i,x_j])=\sum_{k\leq l}\left([u_{ik},u_{jl}]-[u_{jk},u_{il}]\right)\otimes\left(1-\frac{\delta_{kl}}{2}\right)x_kx_l$$

Now since the variables $\{x_kx_l|k\leq l\}$ are linearly independent, we obtain from this $[u_{ik},u_{jl}]=[u_{jk},u_{il}]$, for any $i,j,k,l$. Moreover, if we apply now the antipode we further obtain $[u_{lj},u_{ki}]=[u_{li},u_{kj}]$, and by relabelling, $[u_{ik},u_{jl}]=[u_{il},u_{jk}]$. We therefore conclude that we have $[u_{ik},u_{jl}]=0$ for any $i,j,k,l$, and this finishes the proof. See \cite{bhg}.
\end{proof}

With the above notion in hand, let us investigate the polygonal spheres. We recall from \cite{ban}, \cite{bgo} that the quantum isometry groups of the 5 main spheres are as follows:
$$\xymatrix@R=12mm@C=12mm{
S^{N-1}_\mathbb R\ar[r]\ar@{~}[d]&S^{N-1}_{\mathbb R,*}\ar[r]\ar@{~}[d]&S^{N-1}_{\mathbb R,+}\ar@{~}[d]&\bar{S}^{N-1}_{\mathbb R,*}\ar[l]\ar@{~}[d]&\bar{S}^{N-1}_\mathbb R\ar[l]\ar@{~}[d]\\
O_N\ar[r]&O_N^*\ar[r]&O_N^+&\bar{O}_N^*\ar[l]&\bar{O}_N\ar[l]}$$

Here $O_N$ is the orthogonal group, $O_N^+$ is its free version constructed in \cite{wa1}, $\bar{O}_N$ is its twist constructed in \cite{bbc}, $O_N^*$ is its half-liberated version studied in \cite{bve}, and $\bar{O}_N^*$ is its twisted half-liberated version constructed in \cite{ban}. We refer to \cite{ban} for full details here.

In the polygonal case now, we begin with the classical case. We use the hyperoctahedral group $H_N$, and its free version $H_N^+$ constructed in \cite{bbc}. We have:

\begin{proposition}
The quantum isometry group of $S^{N-1,d-1}_\mathbb R$ is as follows:
\begin{enumerate}
\item At $d=1$ we obtain the free hyperoctahedral group $H_N^+$.

\item At $d=2,\ldots,N-1$ we obtain the hyperoctahedral group $H_N$.

\item At $d=N$ we obtain the orthogonal group $O_N$.
\end{enumerate}
\end{proposition}

\begin{proof}
Observe first that the sphere $S^{N-1,d-1}_\mathbb R$ appears by definition as a union on $\binom{N}{d}$ copies of the sphere $S^{d-1}_\mathbb R$, one for each choice of $d$ coordinate axes, among the coordinate axes of $\mathbb R^N$. We can write this decomposition as follows, with $I_N=\{1,\ldots,N\}$:
$$S^{N-1,d-1}_\mathbb R=\bigcup_{I\subset I_N,|I|=d}(S^{d-1}_\mathbb R)^I$$

(1) At $d=1$ our sphere is $S^{N-1,0}_\mathbb R=\mathbb Z_2^{\oplus N}$, formed by the endpoints of the $N$ copies of $[-1,1]$ on the coordinate axes of $\mathbb R^N$. Thus by \cite{bbc} the quantum isometry group is $H_N^+$.

(2) Our first claim is that at $d\geq2$, the elements $\{x_ix_j|i\leq j\}$ are linearly independent. Since $S^{N-1,1}_\mathbb R\subset S^{N-1,d}_\mathbb R$, we can restrict attention to the case $d=2$. Here the above decomposition is as follows, where $\mathbb T^{\{i,j\}}$ denote the various copies of $\mathbb T$:
$$S^{N-1,d-1}_\mathbb R=\bigcup_{i<j}\mathbb T^{\{i,j\}}$$

Now since $\{x^2,y^2,xy\}$ are linearly independent over $\mathbb T\subset\mathbb R^2$, we deduce from this that $\{x_ix_j|i\leq j\}$ are linearly independent over $S^{N-1,d-1}_\mathbb R$, and we are done. Thus, our claim is proved, and so Proposition 3.2 (2) above applies, and gives $G^+(X)=G(X)$. 

We are therefore left with proving $G(X)=H_N$, for any $d\in\{2,\ldots,N-1\}$. 

Let us first discuss the case $d=2$. Here any affine isometric action $U\curvearrowright S^{N-1,1}_\mathbb R$ must permute the $\binom{N}{2}$ circles $\mathbb T^I$, so we can write $U(\mathbb T^I)=\mathbb T^{I'}$, for a certain permutation of the indices $I\to I'$. Now since $U$ is bijective, we deduce that for any $I,J$ we have:
$$U\left(\mathbb T^I\cap\mathbb T^J\right)=\mathbb T^{I'}\cap\mathbb T^{J'}$$

Since for $|I\cap J|=0,1,2$ we have $\mathbb T^I\cap\mathbb T^J\simeq\emptyset,\{-1,1\},\mathbb T$, by taking the union over $I,J$ with $|I\cap J|=1$, we deduce that $U(\mathbb Z_2^{\oplus N})=\mathbb Z_2^{\oplus N}$. Thus $U\in H_N$, and we are done. 

In the general case now, $d\in\{2,\ldots,N-1\}$, we can proceed similarly, by recurrence. Indeed, for any subsets $I,J\subset I_N$ with $|I|=|J|=d$ we have:
$$(S^{d-1}_\mathbb R)^I\cap(S^{d-1}_\mathbb R)^J=(S^{|I\cap J|-1}_\mathbb R)^{I\cap J}$$

By using $d\leq N-1$, we deduce that we have the following formula:
$$S^{N-1,d-2}_\mathbb R=\bigcup_{|I|=|J|=d,|I\cap J|=d-1}(S^{|I\cap J|-1}_\mathbb R)^{I\cap J}$$

On the other hand, by using the same argument as in the $d=2$ case, we deduce that the space on the right is invariant, under any affine isometric action on $S^{N-1,d-1}_\mathbb R$. Thus by recurrence we obtain $G(S^{N-1,d-1}_\mathbb R)=G(S^{N-1,d-2}_\mathbb R)=H_N$, and we are done.

(3) At $d=N$ the result is known since \cite{bgo}, with the proof coming from the equality $G^+(X)=G(X)$, deduced from Proposition 3.2 (2), as explained above.
\end{proof}

In order to discuss the twisted case, we recall the following definition, from \cite{bbc}:

\begin{definition}
$\bar{O}_N$ is the quantum group obtained by imposing the  relations
$$ab=\begin{cases}
-ba&{\rm for}\ a\neq b\ {\rm on\ the\ same\ row\ or\ column\ of\ }u\\
ba&{\rm otherwise}
\end{cases}$$
to the standard coordinates $u_{ij}$ of the quantum group $O_N^+$.
\end{definition}

As explained in \cite{bbc}, this quantum group has an interesting noncommutative geometric meaning, because it is the quantum isometry group of the hypercube $\mathbb Z_2^N\subset\mathbb R^N$. Thus, $\bar{O}_N$ is a natural analogue of the hyperoctahedral group $H_N$. However, quite surprisingly, $\bar{O}_N$ is not the free version of $H_N$. The correct free version of $H_N$ is the quantum isometry group $H_N^+$ of the space $\mathbb Z_2^{\oplus N}\subset\mathbb R^N$ formed by the $N$ copies of $\mathbb Z_2\subset\mathbb R$ on the coordinate axes of $\mathbb R^N$, that we already met in Proposition 3.3 (1) above. See \cite{bbc}.

Now back to the polygonal spheres, the study in the twisted case is considerably more difficult than in the classical case, and we have complete results only at $d=1,2,N$. Our statement here, to be enhanced later on only with a few minor results, is:

\begin{theorem}
The quantum isometry group of $\bar{S}^{N-1,d-1}_\mathbb R$ is as follows:
\begin{enumerate}
\item At $d=1$ we obtain the free hyperoctahedral group $H_N^+$.

\item At $d=2$ we obtain the hyperoctahedral group $H_N$.

\item At $d=N$ we obtain the twisted orthogonal group $\bar{O}_N$.
\end{enumerate}
\end{theorem}

\begin{proof}
The idea is to adapt the proof of Proposition 3.3 above:

(1) At $d=1$ we have $\bar{S}^{N-1,0}_\mathbb R=S^{N-1,0}_\mathbb R=\mathbb Z_2^{\oplus N}$, and by Proposition 3.3 (1) above, coming from \cite{bbc}, the corresponding quantum isometry group is indeed $H_N^+$. 

(2) As a first ingredient, we will need the twisted analogue of the trick from \cite{bhg}, explained in the proof of Proposition 3.2 (2) above. This twisted trick was already worked out in \cite{ban}, for the sphere $\bar{S}^{N-1}_\mathbb R$ itself, and the situation is similar for any closed subset $X\subset\bar{S}^{N-1}_\mathbb R$, having the property that the variables $\{x_ix_j|i\leq j\}$ are linearly indepedent. More presisely, our claim is that if $G\subset O_N^+$ acts on $X$, then we must have $G\subset\bar{O}_N$.

Indeed, given a coaction $\Phi(x_i)=\sum_ju_{ij}\otimes x_j$, we can write, as in \cite{ban}:
$$\Phi(x_ix_j)=\sum_ku_{ik}u_{jk}\otimes x_k^2+\sum_{k<l}(u_{ik}u_{jl}-u_{il}u_{jk})\otimes x_kx_l$$

We deduce that with $[[a,b]]=ab+ba$ we have the following formula:
$$\Phi([[x_i,x_j]])=\sum_k[[u_{ik},u_{jk}]]\otimes x_k^2+\sum_{k<l}([u_{ik},u_{jl}]-[u_{il},u_{jk}])\otimes x_kx_l$$

Now assuming $i\neq j$, we have $[[x_i,x_j]]=0$, and we therefore obtain $[[u_{ik},u_{jk}]]=0$ for any $k$, and $[u_{ik},u_{jl}]=[u_{il},u_{jk}]$ for any $k<l$. By applying the antipode and then by relabelling, the latter relation gives $[u_{ik},u_{jl}]=0$. Thus we have reached to the defining relations for the quantum group $\bar{O}_N$, and so we have $G\subset\bar{O}_N$, as claimed.

Our second claim is that the above trick applies to any $\bar{S}^{N-1,d-1}_\mathbb R$ with $d\geq2$. Indeed, by using the maps $\pi_{ij}:C(\bar{S}^{N-1,d-1}_\mathbb R)\to C(\bar{S}^1_\mathbb R)$ obtained by setting $x_k=0$ for $k\neq i,j$, we conclude that the variables $\{x_ix_j|i\leq j\}$ are indeed linearly independent over $\bar{S}^{N-1,d-1}_\mathbb R$ .

Summarizing, we have proved so far that if a compact quantum group $G\subset O_N^+$ acts on a polygonal sphere $\bar{S}^{N-1,d-1}_\mathbb R$ with $d\geq2$, then we must have $G\subset\bar{O}_N$. We must now adapt the second part of the proof of Proposition 3.3, and since this is quite unobvious at $d\geq3$, we will restrict now attention to the case $d=2$, as in the statement.

So, consider a compact quantum group $G\subset\bar{O}_N$. In order to have a coaction map $\Phi:C(\bar{S}^{N-1,1}_\mathbb R)\to C(G)\otimes C(\bar{S}^{N-1,1}_\mathbb R)$, given as usual by $\Phi(x_i)=\sum_ju_{ij}\otimes x_j$, the elements $X_i=\sum_ju_{ij}\otimes x_j$ must satisfy the relations $X_iX_jX_k=0$, for any $i,j,k$ distinct.

So, let us compute $X_iX_jX_k$ for $i,j,k$ distinct. We have:
\begin{eqnarray*}
X_iX_jX_k
&=&\sum_{abc}u_{ia}u_{jb}u_{kc}\otimes x_ax_bx_c=\sum_{a,b,c\ not \ distinct}u_{ia}u_{jb}u_{kc}\otimes x_ax_bx_c\\
&=&\sum_{a\neq b}u_{ia}u_{ja}u_{kb}\otimes x_a^2x_b+\sum_{a\neq b}u_{ia}u_{jb}u_{ka}\otimes x_ax_bx_a\\
&+&\sum_{a\neq b}u_{ib}u_{ja}u_{ka}\otimes x_bx_a^2+\sum_au_{ia}u_{ja}u_{ka}\otimes x_a^3
\end{eqnarray*}

By using $x_ax_bx_a=-x_a^2x_b$ and $x_bx_a^2=x_a^2x_b$, we deduce that we have:
\begin{eqnarray*}
X_iX_jX_k
&=&\sum_{a\neq b}(u_{ia}u_{ja}u_{kb}-u_{ia}u_{jb}u_{ka}+u_{ib}u_{ja}u_{ka})\otimes x_a^2x_b+\sum_au_{ia}u_{ja}u_{ka}\otimes x_a^3\\
&=&\sum_{ab}(u_{ia}u_{ja}u_{kb}-u_{ia}u_{jb}u_{ka}+u_{ib}u_{ja}u_{ka})\otimes x_a^2x_b
\end{eqnarray*}

By using now the defining relations for $\bar{O}_N$, which apply to the variables $u_{ij}$, this formula can be written in a cyclic way, as follows:
$$X_iX_jX_k=\sum_{ab}(u_{ia}u_{ja}u_{kb}+u_{ja}u_{ka}u_{ib}+u_{ka}u_{ia}u_{jb})\otimes x_a^2x_b$$

We use now the fact, coming from \cite{ban}, that the variables on the right $x_a^2x_b$ are linearly independent. We conclude that, in order for our quantum group $G\subset\bar{O}_N$ to act on $\bar{S}^{N-1,1}_\mathbb R$, its coordinates must satisfy the following relations, for any $i,j,k$ distinct:
$$u_{ia}u_{ja}u_{kb}+u_{ja}u_{ka}u_{ib}+u_{ka}u_{ia}u_{jb}=0$$

By multiplying to the right by $u_{kb}$ and then by summing over $b$, we deduce from this that we have $u_{ia}u_{ja}=0$, for any $i,j$. Now since the quotient of $C(\bar{O}_N)$ by these latter relations is $C(H_N)$, we conclude that we have $G^+(\bar{S}^{N-1,1}_\mathbb R)=H_N$, as claimed.

(3) At $d=N$ the result is already known from \cite{ban}, and its proof follows in fact from the ``twisted trick'' explained in the proof of (2) above, applied to $\bar{S}^{N-1}_\mathbb R$.
\end{proof}

Observe that the results that we have so far, namely those in \cite{ban} and in Proposition 3.3 and Theorem 3.5 above, give us the quantum isometry groups of 8 of the 9 spheres in Theorem 1.6. The sphere left, $S^{N-1,1}_{\mathbb R,*}$, will be investigated in the next two sections.

In the context of Theorem 3.5, we do not know what happens at $d=3,\ldots,N-1$. It is easy to see that the hyperoctahedral group $H_N$ acts on any polygonal sphere $\bar{S}^{N-1,d-1}_\mathbb R$, and our conjecture would be that this action is the universal one.

\section{Hyperoctahedral groups}

As explained above, our main objective now will be that of computing the quantum isometry group of $S^{N-1,1}_{\mathbb R,*}$. The computation is quite non-trivial, and requires a number of quantum group preliminaries, that we will develop in this section.

We recall from \cite{bcs}, \cite{bve} that the quantum group $O_N^*\subset O_N^+$ is obtained by imposing the half-commutation relations $abc=cba$ to the standard coordinates $u_{ij}$. This quantum group has a twist $\bar{O}_N^*$, constructed in \cite{ban}, whose definition is as follows:

\begin{definition}
$\bar{O}_N^*\subset O_N^+$ is the quantum group obtained by imposing the relations
$$abc=\begin{cases}
-cba&{\rm for\ }r\leq2,s=3{\rm\ or\ }r=3,s\leq2\\
cba&{\rm for\ }r\leq2,s\leq 2{\rm\ or\ }r=s=3
\end{cases}$$
where $r,s\in\{1,2,3\}$ are the number of rows/columns of $u$ spanned by $a,b,c\in\{u_{ij}\}$.
\end{definition}

In order to deal with $\bar{O}_N^*$, it is useful to keep in mind the following table, encoding the choice of the above half-commutation/half-anticommutation signs:
$$\begin{matrix}
r\backslash s&1&2&3\\
1&+&+&-\\
2&+&+&-\\
3&-&-&+
\end{matrix}$$

We have intersected twisted and untwisted spheres in section 2 above, and we will do the same now for the corresponding orthogonal groups. We have here:

\begin{proposition}
The main $5$ quantum groups, and the intersections between them, are
$$
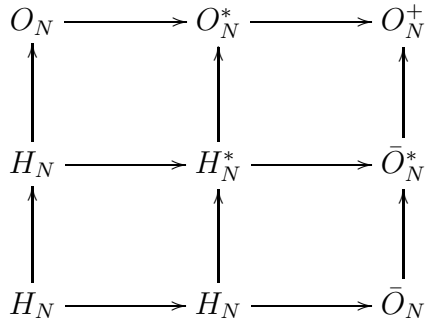
\xymatrix@R=12mm@C=16mm{
O_N\ar[r]&O_N^*\ar[r]&O_N^+\\
H_N\ar[r]\ar[u]&H_N^*\ar[r]\ar[u]&\bar{O}_N^*\ar[u]\\
H_N\ar[r]\ar[u]&H_N\ar[r]\ar[u]&\bar{O}_N\ar[u]}$$
at $N\geq 3$. At $N=2$ the same holds, with the lower left square being $\left[{\,}^{O_2^{\ }}_{H_2^{\ }}{\ }^{O_2^+}_{\bar{O}_2}\right]$.
\end{proposition}

\begin{proof}
We have to study 4 quantum group intersections, as follows:

(1) $O_N\cap\bar{O}_N$. Here an element $U\in O_N$ belongs to the intersection when its entries satisfy $ab=0$ for any $a\neq b$ on the same row or column of $U$. But this means that our matrix $U\in O_N$ must be monomial, and so we get $U\in H_N$, as claimed.

(2) $O_N\cap\bar{O}_N^*$. At $N=2$ the defining relations for $\bar{O}_N^*$ dissapear, and so we have $O_2\cap\bar{O}_2^*=O_2\cap O_2^+=O_2$, as claimed. At $N\geq3$ now, the inclusion $H_N\subset O_N\cap\bar{O}_N^*$ is clear. In order to prove the converse inclusion, pick $U\in O_N$ in the intersection, and assume that $U$ is not monomial. By permuting the entries we can further assume $U_{11}\neq0,U_{12}\neq0$, and from $U_{11}U_{12}U_{i3}=0$ for any $i$ we deduce that the third column of $U$ is filled with $0$ entries, a contradiction. Thus we must have $U\in H_N$, as claimed.

(3) $O_N^*\cap\bar{O}_N$. At $N=2$ we have $O_2^*\cap\bar{O}_2=O_2^+\cap\bar{O}_2=\bar{O}_2$, as claimed. At $N\geq3$ now, best is to use the result in (4) below. Indeed, knowing $O_N^*\cap\bar{O}_N^*=H_N^*$, our intersection is then $G=H_N^*\cap\bar{O}_N$. Now since the standard coordinates on $H_N^*$ satisfy $ab=0$ for $a\neq b$ on the same row or column of $u$, the commutation/anticommutation relations defining $\bar{O}_N$ reduce to plain commutation relations. Thus $G$ follows to be classical, $G\subset O_N$, and by using (1) above we obtain $G=H_N^*\cap\bar{O}_N\cap O_N=H_N^*\cap H_N=H_N$, as claimed.

(4) $O_N^*\cap\bar{O}_N^*$. The result here is non-trivial, and we must use technology from \cite{bdu}. The quantum group $H_N^\times=O_N^*\cap\bar{O}_N^*$ is indeed half-classical in the sense of \cite{bdu}, and since we have $H_N^*\subset H_N^\times$, this quantum group is not classical. Thus the main result in \cite{bdu} applies, and shows that $H_N^\times\subset O_N^*$ must come, via the crossed product construction there, from an intermediate compact group $\mathbb T\subset G\subset U_N$. Now observe that the standard coordinates on $H_N^\times$ are by definition subject to the conditions $abc=0$ when $(r,s)=(\leq2,3),(3,\leq2)$, with the conventions in Definition 4.1 above. It follows that the standard coordinates on $G$ are subject to the conditions $\alpha\beta\gamma=0$ when $(r,s)=(\leq2,3),(3,\leq2)$, where $r,s=span(a,b,c)$ as in Definition 4.1, and $\alpha=a,a^*,\beta=b,b^*,\gamma=c,c^*$. Thus we have $G\subset\bar{U}_N^{**}$, where $\bar{U}_N^{**}$ is the twisted half-liberated version of $U_N$ constructed in \cite{ban}.

We deduce from this that we have $G\subset K_N^\circ$, where $K_N^\circ=U_N\cap\bar{U}_N^{**}$. But this intersection can be computed exactly as in the real case, in the proof of (2) above, and we obtain $K_2^\circ=U_2$, and $K_N^\circ=\mathbb T\wr S_N$ at $N\geq 3$. But the half-liberated quantum groups obtained from $U_2$ and $\mathbb T\wr S_N$ via the construction in \cite{bdu} are well-known, these being $O_2^*=O_2^+$ and $H_N^*$. Thus by functoriality we have $H_2^\times\subset O_2^+$ and $H_N^\times\subset H_N^*$ at $N\geq 3$, and since the reverse inclusions are clear, we obtain $H_2^\times=O_2^+$ and $H_N^\times=H_N^*$ at $N\geq 3$, as claimed.
\end{proof}

Observe that the diagram in Proposition 4.2 is not exactly the quantum isometry group diagram from the introduction. In order to evolve now towards that diagram, we must first introduce a new quantum group, $H_N^{[\infty]}$. This quantum group was constructed in \cite{bcs}, and its main properties, worked out in \cite{bcs}, \cite{rw1}, \cite{rw2}, can be summarized as follows:

\begin{proposition}
Let $H_N^{[\infty]}\subset O_N^+$ be the compact quantum group obtained via the relations $abc=0$, whenever $a\neq c$ are on the same row or column of $u$. 
\begin{enumerate}
\item We have inclusions $H_N^*\subset H_N^{[\infty]}\subset H_N^+$.

\item We have $ab_1\ldots b_rc=0$, whenever $a\neq c$ are on the same row or column of $u$.

\item We have $ab^2=b^2a$, for any two entries $a,b$ of $u$.
\end{enumerate}
\end{proposition}

\begin{proof}
We briefly recall the proof in \cite{bcs}, \cite{rw1}, \cite{rw2}, for future use in what follows. Our first claim is that $H_N^{[\infty]}$ comes, as an easy quantum group, from the following diagram:
$$\xymatrix@R=5mm@C=0.1mm{&\\\pi\ \ =\\&}\xymatrix@R=5mm@C=5mm{
\circ\ar@{-}[dd]&\circ\ar@{.}[dd]&\circ\ar@{-}[dd]\\
\ar@{-}[rr]&&\\
\circ&\circ&\circ}$$

Indeed, since this diagram acts via the map $T_\pi(e_{ijk})=\delta_{ik}e_{ijk}$, we obtain:
\begin{eqnarray*}
T_\pi u^{\otimes 3}e_{abc}&=&T_\pi\sum_{ijk}e_{ijk}\otimes u_{ia}u_{jb}u_{kc}=\sum_{ijk}e_{ijk}\otimes\delta_{ik}u_{ia}u_{jb}u_{kc}\\
u^{\otimes 3}T_\pi e_{abc}&=&u^{\otimes 3}\delta_{ac}e_{abc}=\sum_{ijk}e_{ijk}\otimes\delta_{ac}u_{ia}u_{jb}u_{kc}
\end{eqnarray*}

Thus $T_\pi\in End(u^{\otimes 3})$ is equivalent to $(\delta_{ik}-\delta_{ac})u_{ia}u_{jb}u_{kc}=0$. The non-trivial cases are $i=k,a\neq c$ and $i\neq k,a=c$, and these produce the relations $u_{ia}u_{jb}u_{ic}=0$ for any $a\neq c$, and $u_{ia}u_{jb}u_{ka}=0$, for any $i\neq k$. Thus, we have reached to the relations for $H_N^{[\infty]}$.

(1) The fact that we have inclusions $H_N^*\subset H_N^{[\infty]}\subset H_N^+$ comes from:
$$\xymatrix@R=5mm@C=5mm{
\circ\ar@{-}[ddrr]&\circ\ar@{-}[dd]&\circ\ar@{-}[ddll]\ar@/^/@{.}[r]&\circ\ar@{-}[dd]&\circ\ar@{-}[dd]\\
&&&\ar@{-}[r]&\\
\circ&\circ&\circ\ar@/_/@{.}[r]&\circ&\circ}\ \ \xymatrix@R=5mm@C=1mm{&\\=\\&}\xymatrix@R=5mm@C=5mm{
\circ\ar@{-}[dd]&\circ\ar@{.}[dd]&\circ\ar@{-}[dd]\\
\ar@{-}[rr]&&\\
\circ&\circ&\circ}
\quad\xymatrix@R=5mm@C=1mm{&\\,\\&}\quad
\xymatrix@R=5mm@C=5mm{
\circ\ar@{-}[dd]&\circ\ar@{.}[dd]\ar@/^/@{.}[r]&\circ\ar@{-}[dd]\\
\ar@{-}[rr]&&\\
\circ&\circ\ar@/_/@{.}[r]&\circ}\ \ \xymatrix@R=5mm@C=1mm{&\\=\\&}\xymatrix@R=5mm@C=5mm{
\circ\ar@{-}[dd]&\circ\ar@{-}[dd]\\
\ar@{-}[r]&&\\
\circ&\circ}$$

(2) At $r=2$, the relations $ab_1b_2c=0$ come indeed from the following diagram:
$$\xymatrix@R=5mm@C=5mm{
\circ\ar@{-}[dd]&\circ\ar@{.}[dd]&\circ\ar@{-}[dd]\ar@/^/@{.}[r]&\circ\ar@{-}[dd]&\circ\ar@{.}[dd]&\circ\ar@{-}[dd]\\
\ar@{-}[rr]&&&\ar@{-}[rr]&&\\
\circ&\circ&\circ\ar@/_/@{.}[r]&\circ&\circ&\circ}
\ \ \xymatrix@R=5mm@C=1mm{&\\=\\&}
\xymatrix@R=5mm@C=5mm{
\circ\ar@{-}[dd]&\circ\ar@{.}[dd]&\circ\ar@{.}[dd]&\circ\ar@{-}[dd]\\
\ar@{-}[rrr]&&&\\
\circ&\circ&\circ&\circ}$$

In the general case $r\geq2$ the proof is similar, see \cite{bcs} for details.

(3) We use here an idea from \cite{rw1}, \cite{rw2}. By rotating $\pi$, we obtain:
$$\xymatrix@R=5mm@C=5mm{
\circ\ar@{-}[dd]&\circ\ar@{.}[dd]&\circ\ar@{-}[dd]\\
\ar@{-}[rr]&&\\
\circ&\circ&\circ}
\ \ \xymatrix@R=5mm@C=0.1mm{&\\ \to\\&}\ 
\xymatrix@R=5mm@C=5mm{
\circ\ar@{-}[dd]&\circ\ar@{.}[dd]&\\
\ar@{-}[rrr]&&\ar@{-}[d]&\ar@{-}[d]\\
\circ&\circ&\circ&\circ}
\ \ \xymatrix@R=5mm@C=0.1mm{&\\ \to\\&}\ 
\xymatrix@R=5mm@C=5mm{
\circ\ar@{-}[d]&\circ\ar@{-}[dd]&\circ\ar@{.}[ddll]\\
\ar@{-}[rr]&&\ar@{-}[d]\\
\circ&\circ&\circ}$$

Let us denote by $\sigma$ the partition on the right. Since $T_\sigma(e_{ijk})=\delta_{ij}e_{kji}$, we obtain:
\begin{eqnarray*}
T_\sigma u^{\otimes 3}e_{abc}&=&T_\sigma\sum_{ijk}e_{ijk}\otimes u_{ia}u_{jb}u_{kc}=\sum_{ijk}e_{kji}\otimes\delta_{ij}u_{ia}u_{jb}u_{kc}\\
u^{\otimes 3}T_\sigma e_{abc}&=&u^{\otimes 3}\delta_{ab}e_{cba}=\sum_{ijk}e_{kji}\otimes\delta_{ab}u_{kc}u_{jb}u_{ia}
\end{eqnarray*}

Thus $T_\sigma\in End(u^{\otimes 3})$ is equivalent to $\delta_{ij}u_{ia}u_{jb}u_{kc}=\delta_{ab}u_{kc}u_{jb}u_{ia}$, and now by setting $j=i,b=a$ we obtain the commutation relation $u_{ia}^2u_{kc}=u_{kc}u_{ia}^2$ in the statement.
\end{proof}

The relation of $H_N^{[\infty]}$ with the polygonal spheres comes from the following fact:

\begin{proposition}
Let $X\subset S^{N-1}_{\mathbb R,+}$ be closed, let $d\geq2$, and set $X^{d-1}=X\cap S^{N-1,d-1}_{\mathbb R,+}$. Then for a quantum group $G\subset H_N^{[\infty]}$ the following are equivalent:
\begin{enumerate}
\item $x_i\to\sum_ju_{ij}\otimes x_j$ defines a coaction $\Phi:C(X^{d-1})\to C(G)\otimes C(X^{d-1})$.

\item $x_i\to\sum_ju_{ij}\otimes x_j$ defines a morphism $\widetilde{\Phi}:C(X)\to C(G)\otimes C(X^{d-1})$.
\end{enumerate} 
In particular, $G^+(X)\cap H_N^{[\infty]}$ acts on $X^{d-1}$, for any $d\geq2$.
\end{proposition}

\begin{proof}
The idea here is to use the relations in Proposition 4.3 (2) above:

$(1)\implies(2)$ This is clear, by composing $\Phi$ with the projection map $C(X)\to C(X^{d-1})$. 

$(2)\implies(1)$ In order for a coaction $C(X^{d-1})\to C(G)\otimes C(X^{d-1})$ to exist, the variables $X_i=\sum_ju_{ij}\otimes x_j$ must satisfy the relations defining $X$, which hold indeed by (2), and must satisfy as well the relations $X_{i_0}\ldots X_{i_d}=0$  for $i_0,\ldots,i_d$ distinct, which define $S^{N-1,d-1}_{\mathbb R,+}$. 

The point now is that, under the assumption $G\subset H_N^{[\infty]}$, these latter relations are automatic. Indeed, by using Proposition 4.3 (2), for $i_0,\ldots,i_d$ distinct we obtain:
\begin{eqnarray*}
X_{i_0}\ldots X_{i_d}
&=&\sum_{j_0\ldots j_d}u_{i_0j_0}\ldots u_{i_dj_d}\otimes x_{j_0}\ldots x_{j_d}\\
&=&\sum_{j_0\ldots j_d\ distinct}u_{i_0j_0}\ldots u_{i_dj_d}\otimes 0\ \ +\!\!\!\!\sum_{j_0\ldots j_d\ not\ distinct}0\otimes x_{j_0}\ldots x_{j_d}\\
&=&0+0=0
\end{eqnarray*}

Thus the coaction in (1) exists precisely when (2) is satisfied, and we are done.

Finally, the last assertion is clear from $(2)\implies(1)$, because the universal coaction of $G=G^+(X)$ gives rise to a map $\widetilde{\Phi}:C(X)\to C(G)\otimes C(X^{d-1})$ as in (2).
\end{proof}

As an illustration, we have the following result:

\begin{theorem}
$H_N,H_N,H_N^*,H_N^*,H_N^{[\infty]}$ act respectively on the spheres
$$S^{N-1,d-1}_{\mathbb R},\bar{S}^{N-1,d-1}_{\mathbb R},S^{N-1,d-1}_{\mathbb R,*},\bar{S}^{N-1,d-1}_{\mathbb R,*},S^{N-1,d-1}_{\mathbb R,+}$$ 
at any $d\geq2$.
\end{theorem}

\begin{proof}
We use Proposition 4.4. We know from \cite{ban} that the quantum isometry groups at $d=N$ are respectively equal to $O_N,\bar{O}_N,O_N^*,\bar{O}_N^*,O_N^+$, and our claim is that, by intersecting with $H_N^{[\infty]}$, we obtain the quantum groups in the statement. Indeed:

(1) $O_N\cap H_N^{[\infty]}=H_N$ is clear from definitions.

(2) $\bar{O}_N\cap H_N^{[\infty]}=H_N$ follows from $\bar{O}_N\cap H_N^+\subset O_N$, which in turn follows from the computation (3) in the proof of Proposition 4.2, with $H_N^*$ replaced by $H_N^+$.

(3) $O_N^*\cap H_N^{[\infty]}=H_N^*$ follows from $O_N^*\cap H_N^+=H_N^*$.

(4) $\bar{O}_N^*\cap H_N^{[\infty]}\supset H_N^*$ is clear, and the reverse inclusion can be proved by a direct computation, similar to the computation (3) in the proof of Proposition 4.2.

(5) $O_N^+\cap H_N^{[\infty]}=H_N^{[\infty]}$ is clear from definitions.
\end{proof}

Observe that Theorem 4.5 is sharp, in the sense that the actions there are the universal ones, in the classical case at any $d\in\{2,\ldots,N-1\}$, as well as in the twisted case at $d=2$. Indeed, this follows from Proposition 3.3 and from Theorem 3.5 above.

\section{Quantum isometries}

In this section we complete the computation of the quantum isometry groups of the 9 main spheres, as to prove our main result, announced in the introduction. As already pointed out, we already have results for 8 spheres, the sphere left being $S^{N-1,1}_{\mathbb R,*}$.

We already know from Theorem 4.5 that the quantum group $H_N^*$ from \cite{bcs} acts on $S^{N-1,1}_{\mathbb R,*}$. This action, however, is not universal, because we have:

\begin{proposition}
$\widehat{\mathbb Z_2^{*N}}$ acts on $S^{N-1,1}_{\mathbb R,*}$.
\end{proposition}

\begin{proof}
The standard coordinates on $S^{N-1,1}_{\mathbb R,*}$ are subject to the following relations:
$$x_ix_jx_k=\begin{cases}
0&{\rm for}\ i,j,k\ {\rm distinct}\\
x_kx_jx_i&{\rm otherwise}
\end{cases}$$

Thus, in order to have a coaction map $\Phi:C(S^{N-1,1}_{\mathbb R,*})\to C(G)\otimes C(S^{N-1,1}_{\mathbb R,*})$, given by $\Phi(x_i)=\sum_ju_{ij}\otimes x_j$, the variables $X_i=\sum_ju_{ij}\otimes x_j$ must satisfy the above relations.

For the group dual $G=\widehat{\mathbb Z_2^{*N}}$ we have by definition $u_{ij}=\delta_{ij}g_i$, where $g_1,\ldots,g_N$ are the standard generators of $\mathbb Z_2^{*N}$, and we therefore have:
$$X_iX_jX_k=g_ig_jg_k\otimes x_ix_jx_k,\qquad
X_kX_jX_i=g_kg_jg_i\otimes x_kx_jx_i$$
 
Thus the formula $X_iX_kX_k=0$ for $i,j,k$ distinct is clear, and the formula $X_iX_jX_k=X_kX_jX_i$ for $i,j,k$ not distinct requires $g_ig_jg_k=g_kg_jg_i$ for $i,j,k$ not distinct, which is clear as well. Indeed, at $i=j$ this latter relation reduces to $g_k=g_k$, at $i=k$ this relation is trivial, $g_ig_jg_i=g_ig_jg_i$, and at $j=k$ this relation reduces to $g_i=g_i$.
\end{proof}

More generally, we have the following result:

\begin{proposition}
$H_N^{[\infty]}$ acts on $S^{N-1,1}_{\mathbb R,*}$.
\end{proposition}

\begin{proof}
We proceed as in the proof of Theorem 3.5 above. By expanding the formula of $X_iX_jX_k$ and by using the relations for the sphere $S^{N-1,1}_{\mathbb R,*}$, we have:
\begin{eqnarray*}
X_iX_jX_k
&=&\sum_{abc}u_{ia}u_{jb}u_{kc}\otimes x_ax_bx_c=\sum_{a,b,c\ not \ distinct}u_{ia}u_{jb}u_{kc}\otimes x_ax_bx_c\\
&=&\sum_{a\neq b}(u_{ia}u_{ja}u_{kb}+u_{ib}u_{ja}u_{ka})\otimes x_a^2x_b\\
&+&\sum_{a\neq b}u_{ia}u_{jb}u_{ka}\otimes x_ax_bx_a+\sum_au_{ia}u_{ja}u_{ka}\otimes x_a^3
\end{eqnarray*}

Now by assuming $G=H_N^{[\infty]}$, and by using the various formulae in Proposition 4.3 above, we obtain, for any $i,j,k$ distinct:
$$X_iX_jX_k=\sum_{a\neq b}(0\cdot u_{kb}+u_{ib}\cdot 0)\otimes x_a^2x_b+\sum_{a\neq b}0\otimes x_ax_bx_a+\sum_a(0\cdot u_{ka})\otimes x_a^3=0$$

It remains to prove that we have $X_iX_jX_k=X_kX_jX_i$, for $i,j,k$ not distinct. By replacing $i\leftrightarrow k$ in the above formula of $X_iX_jX_k$, we obtain:
\begin{eqnarray*}
X_kX_jX_i
&=&\sum_{a\neq b}(u_{ka}u_{ja}u_{ib}+u_{kb}u_{ja}u_{ia})\otimes x_a^2x_b\\
&+&\sum_{a\neq b}u_{ka}u_{jb}u_{ia}\otimes x_ax_bx_a+\sum_au_{ka}u_{ja}u_{ia}\otimes x_a^3
\end{eqnarray*}

Let us compare this formula with the above formula of $X_iX_jX_k$. The last sum being 0 in both cases, we must prove that for any $i,j,k$ not distinct and any $a\neq b$ we have:
$$u_{ia}u_{ja}u_{kb}+u_{ib}u_{ja}u_{ka}=u_{ka}u_{ja}u_{ib}+u_{kb}u_{ja}u_{ia}$$
$$u_{ia}u_{jb}u_{ka}=u_{ka}u_{jb}u_{ia}$$

By symmetry the three cases $i=j,i=k,j=k$ reduce to two cases, $i=j$ and $i=k$. The case $i=k$ being clear, we are left with the case $i=j$, where we must prove:
$$u_{ia}u_{ia}u_{kb}+u_{ib}u_{ia}u_{ka}=u_{ka}u_{ia}u_{ib}+u_{kb}u_{ia}u_{ia}$$
$$u_{ia}u_{ib}u_{ka}=u_{ka}u_{ib}u_{ia}$$

By using $a\neq b$, the first equality reads $u_{ia}^2u_{kb}+0\cdot u_{ka}=u_{ka}\cdot 0+u_{kb}u_{ia}^2$, and since by Proposition 4.3 (3) we have $u_{ia}^2u_{kb}=u_{kb}u_{ia}^2$, we are done. As for the second equality, this reads $0\cdot u_{ka}=u_{ka}\cdot 0$, which is true as well, and this ends the proof.
\end{proof}

We will prove now that the action in Proposition 5.2 is universal. In order to do so, we need to convert the formulae of type $X_iX_jX_k=0$ and $X_iX_jX_k=X_kX_jX_i$ into relations between the quantum group coordinates $u_{ij}$, and this requires a good knowledge of the linear relations between the variables $x_a^2x_b,x_ax_bx_a,x_a^3$ over the sphere $S^{N-1,1}_{\mathbb R,*}$.

So, we must first study these variables. The answer here is given by:

\begin{lemma}
The variables $\{x_a^2x_b,x_ax_bx_a,x_a^3|a\neq b\}$ are linearly independent over the sphere $S^{N-1,1}_{\mathbb R,*}$.
\end{lemma}

\begin{proof}
We use a trick from \cite{bdu}. Consider the 1-dimensional polygonal version of the complex sphere $S^{N-1}_\mathbb C$, which is by definition given by:
$$S^{N-1,1}_\mathbb C=\left\{z\in S^{N-1}_\mathbb C\Big|z_iz_jz_k=0,\forall i,j,k\ {\rm distinct}\right\}$$

We have then a $2\times2$ matrix model for the coordinates of $S^{N-1,1}_{\mathbb R,*}$, as follows:
$$x_i\to\gamma_i=\begin{pmatrix}0&z_i\\ \bar{z}_i&0\end{pmatrix}$$

Indeed, the matrices $\gamma_i$ on the right are all self-adjoint, their squares sum up to $1$, they half-commute, and they satisfy $\gamma_i\gamma_j\gamma_k=0$ for $i,j,k$ distinct. Thus we have indeed a morphism $C(S^{N-1,1}_{\mathbb R,*})\to M_2(C(S^{N-1,1}_\mathbb C))$ mapping $x_i\to\gamma_i$, as claimed.

We can use this model in order to prove the linear independence. Indeed, the variables $x_a^2x_b,x_ax_bx_a,x_a^3$ that we are interested in are mapped to the following variables:
$$\gamma_a^2\gamma_b=\begin{pmatrix}0&|z_a|^2z_b\\ |z_a|^2\bar{z}_b&0\end{pmatrix},\quad\gamma_a\gamma_b\gamma_a=\begin{pmatrix}0&z_a^2\bar{z}_b\\ \bar{z}_a^2z_b&0\end{pmatrix},\quad\gamma_a^3=\begin{pmatrix}0&|z_a|^2z_a\\ |z_a|^2\bar{z}_a&0\end{pmatrix}$$

Now since $|z_1|^2z_2,|z_2|^2z_1,z_1^2\bar{z}_2,z_2^2\bar{z}_1,|z_1|^2z_1,|z_2|^2z_2$ are linearly independent over $S^1_\mathbb C$, the upper right entries of the above matrices are linearly independent over $S^{N-1,1}_\mathbb C$. Thus the matrices themselves are linearly independent, and this proves our result.
\end{proof}

With the above lemma in hand, we can now reformulate the coaction problem into a purely quantum group-theoretical problem, as follows:

\begin{lemma}
A quantum group $G\subset O_N^+$ acts on $S^{N-1,1}_{\mathbb R,*}$ precisely when its standard coordinates $u_{ij}$ satisfy the following relations:
\begin{enumerate}
\item $u_{ia}u_{ja}u_{kb}+u_{ib}u_{ja}u_{ka}=0$ for any $i,j,k$ distinct.

\item $u_{ia}u_{jb}u_{ka}=0$ for any $i,j,k$ distinct.

\item $u_{ia}^2u_{kb}=u_{kb}u_{ia}^2$.

\item $u_{ka}u_{ia}u_{ib}=u_{ib}u_{ia}u_{ka}$.

\item $u_{ia}u_{ib}u_{ka}=u_{kb}u_{ib}u_{ia}$.
\end{enumerate}
\end{lemma}

\begin{proof}
We use notations from the beginning of the proof of Proposition 5.2, along with the following formula, also established there:
\begin{eqnarray*}
X_iX_jX_k
&=&\sum_{a\neq b}(u_{ia}u_{ja}u_{kb}+u_{ib}u_{ja}u_{ka})\otimes x_a^2x_b\\
&+&\sum_{a\neq b}u_{ia}u_{jb}u_{ka}\otimes x_ax_bx_a+\sum_au_{ia}u_{ja}u_{ka}\otimes x_a^3
\end{eqnarray*}

In order to have an action as in the statement, these quantities must satisfy $X_iX_kX_k=0$ for $i,j,k$ disctinct, and $X_iX_kX_k=X_kX_jX_i$ for $i,j,k$ not distinct. Now by using Lemma 5.3, we conclude that the relations to be satisfied are as follows:

(A) For $i,j,k$ distinct, the following must hold:
$$u_{ia}u_{ja}u_{kb}+u_{ib}u_{ja}u_{ka}=0,\forall a\neq b$$
$$u_{ia}u_{jb}u_{ka}=0,\forall a\neq b$$
$$u_{ia}u_{ja}u_{ka}=0,\forall a$$

(B) For $i,j,k$ not distinct, the following must hold:
$$u_{ia}u_{ja}u_{kb}+u_{ib}u_{ja}u_{ka}=u_{ka}u_{ja}u_{ib}+u_{kb}u_{ja}u_{ia},\forall a\neq b$$
$$u_{ia}u_{jb}u_{ka}=u_{ka}u_{jb}u_{ia},\forall a\neq b$$
$$u_{ia}u_{ja}u_{ka}=u_{ka}u_{ja}u_{ia},\forall a$$

In order to simplify this set of relations, the first observation is that the last relations in both (A) and (B) can be merged with the other ones, and we are led to:

(A') For $i,j,k$ distinct, the following must hold:
$$u_{ia}u_{ja}u_{kb}+u_{ib}u_{ja}u_{ka}=0,\forall a,b$$
$$u_{ia}u_{jb}u_{ka}=0,\forall a,b$$

(B') For $i,j,k$ not distinct, the following must hold:
$$u_{ia}u_{ja}u_{kb}+u_{ib}u_{ja}u_{ka}=u_{ka}u_{ja}u_{ib}+u_{kb}u_{ja}u_{ia},\forall a,b$$
$$u_{ia}u_{jb}u_{ka}=u_{ka}u_{jb}u_{ia},\forall a,b$$

Observe that the relations (A') are exactly the relations (1,2) in the statement.

Let us further process the relations (B'). In the case $i=k$ the relations are automatic, and in the cases $j=i,j=k$ the relations that we obtain coincide, via $i\leftrightarrow k$. Thus (B') reduces to the set of relations obtained by setting $j=i$, which are as follows:
$$u_{ia}u_{ia}u_{kb}+u_{ib}u_{ia}u_{ka}=u_{ka}u_{ia}u_{ib}+u_{kb}u_{ia}u_{ia}$$
$$u_{ia}u_{ib}u_{ka}=u_{ka}u_{ib}u_{ia}$$

Observe that the second relation is the relation (5) in the statement. Regarding now the first relation, with the notation $[x,y,z]=xyz-zyx$, this is as follows:
$$[u_{ia},u_{ia},u_{kb}]=[u_{ka},u_{ia},u_{ib}]$$

By applying the antipode, we obtain $[u_{bk},u_{ai},u_{ai}]=[u_{bi},u_{ai},u_{ak}]$, and then relabelling $a\leftrightarrow i$ and $b\leftrightarrow k$, this relation becomes $[u_{kb},u_{ia},u_{ia}]=[u_{ka},u_{ia},u_{ib}]$. Now since we have $[x,y,z]=-[z,y,x]$, by comparing this latter relation with the original one, a simplification occurs, and the resulting relations are as follows:
$$[u_{ia},u_{ia},u_{kb}]=[u_{ka},u_{ia},u_{ib}]=0$$

But these are exactly the relations (3,4) in the statement, and we are done.
\end{proof}

Now by solving the quantum group problem raised by Lemma 5.4, we obtain:

\begin{proposition}
We have $G^+(S^{N-1,1}_{\mathbb R,*})=H_N^{[\infty]}$.
\end{proposition}

\begin{proof}
The inclusion $\supset$ is clear from Proposition 5.2. For the converse, we already have the result at $N=2$, so assume $N\geq3$. By using Lemma 5.4 (2), for $i\neq j$ we have:
\begin{eqnarray*}
u_{ia}u_{jb}u_{ka}=0,\forall k\neq i,j
&\implies&u_{ia}u_{jb}u_{ka}^2=0,\forall k\neq i,j\\
&\implies&u_{ia}u_{jb}\left(\sum_{k\neq i,j}u_{ka}^2\right)=0,\forall i\neq j\\
&\implies&u_{ia}u_{jb}(1-u_{ia}^2-u_{ja}^2)=0,\forall i\neq j
\end{eqnarray*}

Now by using Lemma 5.4 (3), we can move the variable $u_{jb}$ to the right. By further multiply by $u_{jb}$ to the right, and then summing over $b$, we obtain:
\begin{eqnarray*}
u_{ia}u_{jb}(1-u_{ia}^2-u_{ja}^2)=0,\forall i\neq j
&\implies&u_{ia}(1-u_{ia}^2-u_{ja}^2)u_{jb}=0,\forall i\neq j\\
&\implies&u_{ia}(1-u_{ia}^2-u_{ja}^2)u_{jb}^2=0,\forall i\neq j\\
&\implies&u_{ia}(1-u_{ia}^2-u_{ja}^2)=0,\forall i\neq j
\end{eqnarray*}

We can proceed now as follows, by summing over $j\neq i$:
\begin{eqnarray*}
u_{ia}(1-u_{ia}^2-u_{ja}^2)=0,\forall i\neq j
&\implies&u_{ia}u_{ja}^2=u_{ia}-u_{ia}^3,\forall i\neq j\\
&\implies&u_{ia}(1-u_{ia}^2)=(N-1)(u_{ia}-u_{ia}^3)\\
&\implies&u_{ia}=u_{ia}^3
\end{eqnarray*}

Thus the standard coordinates are partial isometries, and so $G\subset H_N^+$. On the other hand, we know from the proof of Proposition 4.3 (3) that the quantum subgroup $G\subset H_N^+$ obtained via the relations $[a,b^2]=0$ is $H_N^{[\infty]}$, and this finishes the proof.
\end{proof}

We have now complete results for the 9 main spheres, as follows:

\begin{theorem}
The quantum isometry groups of the $9$ main spheres are
$$\xymatrix@R=13mm@C=17mm{
O_N\ar[r]&O_N^*\ar[r]&O_N^+\\
H_N\ar[r]\ar[u]&H_N^{[\infty]}\ar[r]\ar[u]&\bar{O}_N^*\ar[u]\\
H_N^+\ar[r]\ar[u]&H_N\ar[r]\ar[u]&\bar{O}_N\ar[u]}$$
where $H_N^+,H_N^{[\infty]}$ and $\bar{O}_N,O_N^*,\bar{O}_N^*,O_N^*$ are noncommutative versions of $H_N,O_N$.
\end{theorem}

\begin{proof}
This follows indeed from \cite{ban}, \cite{bgo} and from the above results.
\end{proof}

As a first comment, in view of the conjectures in section 2 above, Theorem 5.6 probably deals with the general mixed monomial case. We do not know if it is so.

In general, there are of course many questions left. Perhaps the very first question here regards $S^{N-1,1}_{\mathbb R,+}$, whose quantum isometry group should be probably $H_N^{[\infty]}$. Technically speaking, the problem is that we have no good models for $S^{N-1,1}_{\mathbb R,+}$, and hence no tools for dealing with independence questions for products of coordinates over it.

We should remind, however, that $S^{N-1,1}_{\mathbb R,+}$ is a bit of a ``pathological'' sphere. Besides various issues with diagrams and axiomatization, coming from sections 1-2 above, one problem is that the operation $S^{N-1,1}_\mathbb R\to S^{N-1,1}_{\mathbb R,+}$ is not exactly a ``liberation'' in the sense of free probability theory \cite{bpa}, \cite{vdn}. More precisely, as explained, in \cite{bgo}, the operation $S^{N-1}_\mathbb R\to S^{N-1}_{\mathbb R,+}$ is compatible with the Bercovici-Pata bijection \cite{bpa}, at the level of the corresponding hyperspherical laws, but this seems to fail for $S^{N-1,1}_\mathbb R\to S^{N-1,1}_{\mathbb R,+}$.

Summarizing, if all our conjectures and guesses hold true, Theorem 5.6 above might be indeed the ``final'' statement regarding the quantum isometries of polygonal spheres. Note however that the Riemannian interpretation of our various computations, in the smooth case, in the spirit of the constructions in \cite{bgo}, remains an open problem.

\section{Complexification issues}

In this section we discuss a straightforward complex extension of the above results. Our starting point will be the following definition, from \cite{ban}:

\begin{definition}
We consider the universal $C^*$-algebra
$$C(S^{N-1}_{\mathbb C,+})=C^*\left(z_1,\ldots,z_N\Big|\sum_iz_iz_i^*=\sum_iz_i^*z_i=1\right)$$
and call the underlying space $S^{N-1}_{\mathbb C,+}$ free complex sphere.
\end{definition}

As a first observation, the relation between the real and complex spheres is quite unobvious in the free case. Recall indeed that in the classical case we have an isomorphism $S^{N-1}_\mathbb C\simeq S^{2N-1}_\mathbb R$, obtained by setting $z_i=x_i+iy_i$. In the free case no such isomorphism is available, and in fact both inclusions $S^{N-1}_{\mathbb C,+}\subset S^{2N-1}_{\mathbb R,+}$ and $S^{2N-1}_{\mathbb R,+}\subset S^{N-1}_{\mathbb C,+}$ fail to hold. This is due to the formula $(x+iy)(x-iy)=(x^2+y^2)-i[x,y]$, which makes appear the commutator $[x,y]$, which has no reasons to vanish for the free spherical coordinates.

We can define quantum isometry groups, in a complex sense, as follows:

\begin{definition}
Consider an algebraic manifold $Z\subset S^{N-1}_{\mathbb C,+}$, assumed to be non-degenerate, in the sense that the variables $Re(z_i),Im(z_i)\in C(Z)$ are linearly independent.
\begin{enumerate}
\item We let $G^+(Z)\subset U_N^+$ be the biggest quantum subgroup acting affinely on $Z$.

\item We set as well $G(Z)=G^+(Z)\cap U_N$, with the intersection taken inside $U_N^+$.
\end{enumerate}
\end{definition}

Here $U_N^+$ is the free analogue of $U_N$, constructed by Wang in \cite{wa1}, and the existence and uniqueness of $G^+(Z)$ follow as in the proof of Proposition 4.1 above.

In the classical case, where $Z\subset S^{N-1}_\mathbb C$, we can use the isomorphism $S^{N-1}_\mathbb C\simeq S^{2N-1}_\mathbb R$ in order to view $Z$ as a real manifold, $Z_r\subset S^{2N-1}_\mathbb R$. We can therefore construct a ``real'' quantum isometry group $G^+(Z_r)\subset O_{2N}^+$, and have $G^+(Z)=G^+(Z_r)\cap U_N^+$, where the intersection is taken inside $O_{2N}^+$, by using the embedding $U_N^+\subset O_{2N}^+$ given by the fact that for $u=v+iw$ biunitary, the matrix $(^v_w{\ }^w_v)$ is orthogonal. See \cite{bdu}.

As an example here, consider the torus $\mathbb T\subset\mathbb C$. A straightforward complex extension of the trick in Proposition 3.2 (2), explained in \cite{ban}, shows that we have $G^+(\mathbb T)=G(\mathbb T)=U_1$. We should mention that it is true as well that we have $G^+(\mathbb T_r)=G(\mathbb T_r)=O_2$, therefore confirming the formula $G^+(\mathbb T)=G^+(\mathbb T_r)\cap U_1^+$, but this result holds due to much deeper reasons, explained by Bhowmick in \cite{bho}. For more on these issues, see also \cite{gjo}.

In the non-classical case, as explained above, there is no embedding $S^{N-1}_{\mathbb C,+}\subset S^{2N-1}_{\mathbb R,+}$ that can be used, and the relation between Proposition 4.1 and Definition 6.2 remains quite unclear. In short, we have to develop the complex theory ``paralleling'' the real one.

As explained in \cite{ban}, the 5 real spheres have 5 complex analogues. We can extend this analogy to the level of polygonal spheres, as follows:

\begin{definition}
The spheres $S^{N-1,d-1}_{\mathbb C},\bar{S}^{N-1,d-1}_{\mathbb C},S^{N-1,d-1}_{\mathbb C,**},\bar{S}^{N-1,d-1}_{\mathbb C,**},S^{N-1,d-1}_{\mathbb C,+}$ are constructed from $S^{N-1}_{\mathbb C,+}$ in the same way as their real counterparts are constructed from $S^{N-1}_{\mathbb R,+}$, by assuming that the corresponding relations hold between the variables $x_i=z_i,z_i^*$.
\end{definition}

Here we use the convention that the subscript $**$ from the complex case corresponds to the subscript $*$ from the real case. For more on this issue, see \cite{ban}, \cite{bdu}.

As an illustration, in the free case the polygonal spheres are as follows:
$$C(S^{N-1,d-1}_{\mathbb C,+})=C(S^{N-1}_{\mathbb C,+})/\left\langle z_{i_0}^{\varepsilon_0}\ldots z_{i_d}^{\varepsilon_d}=0,\forall i_0,\ldots,i_d\ {\rm distinct},\forall \varepsilon_1,\ldots,\varepsilon_d\in\{1,*\}\right\rangle$$

As in the real case, we will restrict now the attention to the 5 main spheres, coming from \cite{ban}, and to their intersections. We have 9 such spheres here, as follows:

\begin{proposition}
The $5$ main spheres, and the intersections between them, are
$$\xymatrix@R=12mm@C=12mm{
S^{N-1}_\mathbb C\ar[r]&S^{N-1}_{\mathbb C,**}\ar[r]&S^{N-1}_{\mathbb C,+}\\
S^{N-1,1}_\mathbb C\ar[r]\ar[u]&S^{N-1,1}_{\mathbb C,**}\ar[r]\ar[u]&\bar{S}^{N-1}_{\mathbb C,**}\ar[u]\\
S^{N-1,0}_\mathbb C\ar[r]\ar[u]&\bar{S}^{N-1,1}_\mathbb C\ar[r]\ar[u]&\bar{S}^{N-1}_\mathbb C\ar[u]}$$
with all the maps being inclusions.
\end{proposition}

\begin{proof}
This is similar to the proof of Proposition 1.2 above, by replacing in all the computations there the variables $x_i$ by the variables $x_i=z_i,z_i^*$.
\end{proof}

As explained in \cite{ban}, the axiomatization problems in the complex case are quite similar to those in the real case, and the same happens in the present polygonal context. Thus, we will not review in detail the material from sections 1-2 above. Let us mention, however, that there are a few subtleties appearing in the complex case. For instance the saturation notion in Definition 2.2 (1) above has a straightforward complex analogue, but it is not clear whether the real and complex saturations of a filtered group $G\subset S_\infty$ coincide. In short, the ``noncommutative algebraic geometry'' questions discussed in sections 1-2 above are expected to be the same over $\mathbb R$ and $\mathbb C$, but we don't have a proof for this fact.

Let us discuss now the computation of the associated quantum isometry groups, following some previous results from \cite{ban}, and the material from sections 3-5 above. 

We use the compact group $K_N=\mathbb T\wr S_N$ and its free version $K_N^+=\mathbb T\wr_*S_N^+$, which appear as straightforward complex analogues of the hyperoctahedral group $H_N=\mathbb Z_2\wr S_N$, and of its free version $H_N^+=\mathbb Z_2\wr_*S_N^+$, constructed in \cite{bbc}. Also, we define the complex version $K_N^{[\infty]}\subset U_N^+$ of the quantum group $H_N^{[\infty]}\subset O_N^+$ by using the relations $\alpha\beta\gamma=0$ with $\alpha=a,a^*,\beta=b,b^*,\gamma=c,c^*$, for any $a\neq c$ on the same row of $u$.

With these conventions, we have the following result:

\begin{theorem}
The quantum isometry groups of the $9$ main complex spheres are
$$\xymatrix@R=13mm@C=17mm{
U_N\ar[r]&U_N^{**}\ar[r]&U_N^+\\
K_N\ar[r]\ar[u]&K_N^{[\infty]}\ar[r]\ar[u]&\bar{U}_N^{**}\ar[u]\\
K_N^+\ar[r]\ar[u]&K_N\ar[r]\ar[u]&\bar{U}_N\ar[u]}$$
where $K_N$ and its versions are the complex analogues of $H_N$ and its versions.
\end{theorem}

\begin{proof}
The idea is that the proof here is quite similar to the proof in the real case, by replacing $H_N,O_N$ with their complex analogues $K_N,U_N$.

More precisely, the results for the 5 spheres on top and on the right are already known from \cite{ban}. Regarding the remaining 4 spheres, the proof here is as follows:

(1) We have $S^{N-1,0}_\mathbb C=\mathbb T^{\oplus N}$, whose quantum isometry group is indeed $K_N^+$. This follows as in \cite{bbc}, by adapting the proof from there of $G^+(\mathbb Z_2^{\oplus N})=H_N^+$.

(2) We have a decomposition $S^{N-1,1}_\mathbb C=\bigcup_{i<j}(S^1_\mathbb C)^{\{i,j\}}$, which is similar to the one in the real case, and the reduction method in the proof of Proposition 3.3 (2) applies, and shows that the quantum isometry group is $K_N^+\cap U_N=K_N$, as claimed.

(3) Regarding $\bar{S}^{N-1,1}_\mathbb C$, the first part of the proof of Theorem 3.5 above extends to the complex case, by using as key ingredient the formula $G^+(\bar{S}^{N-1}_\mathbb C)=\bar{U}_N$ from \cite{ban}. The second part extends as well, by replacing everywhere the variables $x_i$ by the variables $x_i=z_i,z_i^*$, and shows that the quantum isometry group is $K_N$, as claimed.

(4) Finally, regarding $\bar{S}^{N-1,1}_{\mathbb C,**}$, all the computations in the proof of Lemma 5.4 and Proposition 5.5 above extend to the complex case, by replacing everywhere the variables $x_i$ by the variables $x_i=z_i,z_i^*$, and show that the quantum isometry group is $K_N^{[\infty]}$.
\end{proof}

Regarding the remaining complex polygonal spheres, the situation here is quite similar to the one in the real case. Techically speaking, the problem is that Proposition 3.3, whose complex analogue can be shown to fully hold, is quite unobvious to extend.

As a conclusion, at the abstract classification level we have enlaged the set of 10 spheres in \cite{ban} with 8 more spheres, which should be generally regarded as being not smooth. We should mention that the $10\to 18$ extension announced in \cite{ban}, via free complexification, is of course different from the one performed here. The extension via free complexification still remains to be done, but ideally under the present, upgraded formalism.

This adds to the various questions raised throughout the paper.

\end{document}